\documentclass[12pt]{amsart}

\usepackage[marginratio=1:1,scale=0.75]{geometry} 
\usepackage{amsmath}
\usepackage{amssymb}
\usepackage{latexsym}
\usepackage{amsthm}
\usepackage{mathrsfs}
\usepackage{color}
\usepackage{amsfonts}
\usepackage{enumitem}
\usepackage{array}
\usepackage{setspace}
\usepackage{mathtools}
\usepackage{tikz}
\usepackage{tikz-cd}
\usetikzlibrary{graphs,decorations.pathmorphing,decorations.markings}
\usepackage{xcolor}
\usepackage{hyperref}
\hypersetup{
    colorlinks,
    linkcolor={red!70!black},
    citecolor={green!40!black},
    urlcolor={blue!80!black}
}
\usetikzlibrary{arrows}
\usetikzlibrary{knots,patterns}

\usepackage{stmaryrd,graphicx}

\newtheorem{thm}{Theorem}[section]
\newtheorem*{thm*}{Theorem}
\newtheorem{cor}[thm]{Corollary}
\newtheorem*{cor*}{Corollary}
\newtheorem{lem}[thm]{Lemma}
\newtheorem*{lem*}{Lemma}
\newtheorem{prop}[thm]{Proposition}
\newtheorem*{prop*}{Proposition}

\theoremstyle{definition}
\newtheorem{defn}[thm]{Definition}
\newtheorem*{defn*}{Definition}

\newtheorem*{conjecture*}{Conjecture}
\newtheorem*{condition*}{Condition}
\newtheorem*{assumption*}{Assumption}

\theoremstyle{remark}
\newtheorem{rem}[thm]{Remark}
\newtheorem*{rem*}{Remark}

\newtheorem*{problem*}{Problem}

\numberwithin{equation}{section}

\newcommand{\BQ}{\mathbb Q}

\newcommand{\BG}{\mathbb G}
\newcommand{\BC}{\mathbb C}

\newcommand{\BN}{\mathbb N}
\newcommand{\BA}{\mathbb A}

\newcommand{\BZ}{\mathbb Z}
\newcommand{\BP}{\mathbb P}

\newcommand{\Bmu}{{\pmb \mu}}

\newcommand{\CE}{\mathcal E}
\newcommand{\CF}{\mathcal F}
\newcommand{\CG}{\mathcal G}
\newcommand{\CO}{\mathcal O}
\newcommand{\CH}{\mathcal H}

\newcommand{\CT}{\mathcal T}
\newcommand{\CM}{\mathcal M}

\newcommand{\CX}{\mathcal X}

\newcommand{\CZ}{\mathcal Z}
\newcommand{\V}{\mathbf{V}}

\newcommand{\bfF}{\mathbf{F}}

\newcommand{\ev}{\mathrm {ev}}

\newcommand{\wt}{\widetilde}
\newcommand{\ol}{\overline}
\newcommand{\eps}{\epsilon}
\newcommand{\veps}{\varepsilon}

\newcommand{\parab}{\mathrm{par}}

\newcommand{\vir}{\mathrm{vir}}
\newcommand{\pure}{\mathrm{pure}}
\newcommand{\taut}{\mathrm{taut}}
\newcommand{\ellipt}{\mathrm{ell}}
\newcommand{\parell}{\mathrm{parell}}
\newcommand{\red}{\mathrm{red}}

\newcommand{\unred}{\mathrm{unred}}

\newcommand{\Mod}{\mathrm{Mod}}
\newcommand{\frakd}{\mathfrak{d}}
\newcommand{\frakh}{\mathfrak{h}}
\newcommand{\frake}{\mathfrak{e}}
\newcommand{\frakf}{\mathfrak{f}}
\newcommand{\Betti}{\mathrm{Betti}}
\newcommand{\Wred}{W_M}

\DeclareMathOperator{\Gr}{Gr}
\DeclareMathOperator{\Ad}{Ad}
\DeclareMathOperator{\ch}{ch}

\DeclareMathOperator{\GL}{GL}
\DeclareMathOperator{\Sp}{Sp}
\DeclareMathOperator{\liesl}{\mathfrak{sl}}
\DeclareMathOperator{\liegl}{\mathfrak{gl}}
\DeclareMathOperator{\PGL}{PGL}

\DeclareMathOperator{\Hom}{Hom}
\DeclareMathOperator{\Tot}{Tot}
\DeclareMathOperator{\Td}{Td}

\DeclareMathOperator{\Coh}{Coh}

\DeclareMathOperator{\res}{res}

\DeclareMathOperator{\trace}{Tr}

\DeclareMathOperator{\id}{id}

\DeclareMathOperator{\ASym}{ASym}

\DeclareMathOperator{\End}{End}

\DeclareMathOperator{\supp}{supp}

\DeclareMathOperator{\kernel}{Ker}
\DeclareMathOperator{\image}{Im}
\DeclareMathOperator{\cokernel}{Coker}

\DeclareMathOperator{\Id}{Id}

\DeclareMathOperator{\rank}{rk}

\title{$P=W$ via $\CH_2$}

\author{Tam\'as Hausel}
\address[Tam\'as Hausel]{IST Austria, Klosterneuburg, Austria}
\email{tamas.hausel@ist.ac.at}

\author{Anton Mellit}
\address[Anton Mellit]{Faculty of Mathematics, University of Vienna, Vienna, Austria}
\email{anton.mellit@univie.ac.at}

\author{Alexandre Minets}
\address[Alexandre Minets]{Max Planck Institute for Mathematics, Bonn, Germany}
\email{minets@mpim-bonn.mpg.de}
\curraddr{Mathematisches Institut, University of Bonn, Germany}
\email{aminets@math.uni-bonn.de}

\author{Olivier Schiffmann}
\address[Olivier Schiffmann]{Laboratoire de Mathématiques d'Orsay, Université de Paris-Saclay, Orsay, France, and Simion Stoilow Institute of Mathematics, Bucarest, Romania}
\email{olivier.schiffmann@universite-paris-saclay.fr}

\date{\today}

\begin{document}
	\onehalfspacing

	\begin{abstract}
		Let $\CH_2$ be the Lie algebra of polynomial Hamiltonian vector fields on the symplectic plane. Let $X$ be the moduli space of stable Higgs bundles of fixed relatively prime rank and degree, or more generally the moduli space of stable parabolic Higgs bundles of arbitrary rank and degree for a generic stability condition. Let $H^*(X)$ be the cohomology with rational coefficients. Using the operations of cup-product by tautological classes and Hecke correspondences we construct an action of $\CH_2$ on $H^*(X)[x,y]$, where $x$ and $y$ are formal variables. We show that the perverse filtration on $H^*(X)$ coincides with the filtration canonically associated to $\liesl_2\subset\CH_2$ and deduce the $P=W$ conjecture of de Cataldo-Hausel-Migliorini.
	\end{abstract}

	\maketitle
	

	\section{Introduction}
	
	\subsection{The $P=W$ conjecture.}

	Let $r>0$ and $d$ be relatively prime integers and let $C$ be a smooth projective algebraic curve over $\BC$. Following Hitchin \cite{hitchin1987self} and Simpson \cite{simpson1990harmonic,simpson1992higgs}, \emph{Higgs bundles} are pairs $(\CE, \theta)$ such that $\CE$ is an algebraic vector bundle on $C$ and $\theta:\CE\to \CE\otimes \Omega$, where $\Omega$ denotes the line bundle of differential forms on $C$. A Higgs bundle is \emph{stable} if for all proper subbundles $\CE'\subset \CE$ of rank $r'$ and degree $d'$ satisfying $\theta(\CE')\subset \CE'\otimes\Omega$ we have the inequality
	\[
	\frac{d'}{r'}< \frac{d}{r}.
	\]
    Let $M_{r,d}$ be the moduli space of stable Higgs bundles on $C$ of rank $r$ and degree $d$, which is a (smooth) quasi-projective symplectic variety of dimension $2(g-1)r^2+2$. By the non-abelian Hodge correspondence (\cite{hitchin1987self, simpson1991nonabelian, simpson1992higgs}), the space  $M_{r,d}$ is diffeomorphic to the moduli space $M_{r,d}^\Betti$ parameterizing certain twisted local systems on $C$. That latter space is an affine algebraic variety, but the complex structure is different from the one of $M_{r,d}$. By \cite{deligne1971theorie}, the cohomology with complex coefficients $H^*(M_{r,d}^\Betti)$ is equipped with a natural \emph{weight filtration}. It is natural to ask for a description of the weight filtration on $H^*(M_{r,d}) = H^*(M_{r,d}^\Betti)$ in terms of the algebraic geometry of the space $M_{r,d}$.
	
	A conjectural answer to this question was suggested in \cite{cataldo2012topology}. The moduli space $M_{r,d}$ is endowed with the \emph{Hitchin map} $\chi:M_{r,d}\to \BC^{(g-1)r^2+1}$ which reads off the coefficients of the characteristic polynomial of $\theta$. It turns out that $\chi$ is proper and applying the decomposition theorem \cite{beuilinson1982faisceaux} one constructs the perverse filtration, starting in degree $N=(g-1)r^2+1$. The $P=W$ conjecture then claims that the weight filtration and the perverse filtration are essentially equal:
\begin{conjecture*}[$P=W$,~\cite{cataldo2012topology}]
For $i=0,1,\ldots,(g-1)r^2+2$, we have $$P_{i-N} H^*(M_{r,d}) = W_{2i} H^*(M_{r,d}^\Betti).$$
\end{conjecture*}
In this paper, we prove this conjecture as well as its analog in which one replaces stable Higgs bundles by stable parabolic Higgs bundles. Our approach goes through a variant of the $P=W$ conjecture knows as the $P=C$ conjecture. As proven by \cite{markman2002generators}, the cohomology ring $H^*(M_{r,d})$ is generated by the tautological classes, which are the components of the K\"unneth decomposition of the Chern classes of the tautological sheaf on $M_{r,d}\times C$. By \cite{shende2016weights} the weight filtration may be entirely described in terms of tautological classes. Namely, assign a weight to each tautological class so that the classes coming from the $i$-th Chern class have weight $2i$. Then $W_m$ is spanned by products of tautological classes of total weight $\leq m$.
	
Our main result shows that the same kind of description holds for $P$:
\begin{thm*}[$P=C$ conjecture]
The subspace $P_m H^*(M_{r,d})$ is the span of products of tautological classes of total weight $\leq 2(m + N)$, where $N=(g-1)r^2+1$.
\end{thm*}

This immediately implies
\begin{cor*}
The $P=W$ conjecture holds for the spaces $M_{r,d}$.
\end{cor*}

As mentioned above, we work in a more general context of moduli spaces of stable parabolic Higgs bundles of arbitrary rank and degree for a generic stability condition, and prove the corresponding result in that context. In fact, it is essential for our proof to consider both parabolic and non-parabolic setups at once. 
	
The $P=W$ conjecture in the case of rank $2$ was proved already in~\cite{cataldo2012topology}. The case when $C$ has genus $2$ was established in~\cite{cataldo2022hitchin}. While the present paper was in the final stages of preparation, we learned of a proof of $P=W$ by Maulik and Shen using a different approach, see \cite{maulik2022p}. Since then, a third proof has appeared, see~\cite{maulik2023}. We refer the interested reader to \cite{hoskins23} for a survey and comparison of these proofs.

\subsection{The strategy}
The main observation motivating our approach is the following: modifications of Higgs bundles at a point $p\in C$ do not change the spectral curve, or in other words \textit{Hecke correspondences commute with the Hitchin map} $\chi$.
This suggests that the perverse filtration $P$, which reflects the geometry of $\chi$, is nicely compatible with the action of Hecke correspondences.
More precisely, by the relative Hard Lefschetz theorem $P$, together with the operator $\mathfrak{e}$ of cup product with a relatively ample class in $H^2$, form a Lefschetz structure (see Section~\ref{sec:Relative Hard Lefschetz} for details).
Our strategy is to promote this to the action of an $\mathfrak{sl}_2$-triple, which is expressed in terms of Hecke correspondences, and whose weights are precisely the weights of tautological classes.  

The algebra of all (punctual) Hecke correspondences on a given surface $S$ is called the \textit{cohomological Hall algebra} (CoHA) of zero-dimensional sheaves on $S$, and was extensively studied in~\cite{MMSV}.
It was shown that this CoHA is isomorphic to a deformed $W_{1+\infty}$-algebra $W_S$, modeled on the cohomology ring $H^*(S)$.
This presentation drastically simplifies when the surface satisfies some vanishing properties, such as for open symplectic surfaces (which is the case of interest for Higgs bundles).
Moreover, $W_S$ acts on the tautological cohomology classes of various moduli of coherent sheaves on $S$ via explicit formulas, best described in terms of the so-called Fock space representation (Corollary~\ref{cor:surface W}).
We recall these results, together with some proofs included for completeness, in Sections~\ref{sec:hecke} and~\ref{sec:fock space}.

Now let $S = T^*C$ and consider the moduli stack of Higgs bundles $\mathcal{M}_r = \bigsqcup_d \mathcal{M}_{r,d}$.
It turns out that the desired $\mathfrak{sl}_2$-triple cannot be found in $W_{T^*C}$ by itself.
Instead, we need to use certain properties of the $W_S$-action on $\mathcal{M}_r$, such as polynomiality (Proposition~\ref{prop:D m n polynomial}) which allow us,
via a series of \textit{degenerations steps} (see Section~\ref{sec:sl2}) to replace $W_S$ with a smaller algebra $W_M$, which now acts on the cohomology of the coarse moduli space $H^*(M_{r,d})$.
It is this latter algebra which contains the $\mathfrak{sl}_2$-triple which we seek.
Schematically, these steps can be represented as follows; the meaning of the second column is explained in Section~\ref{sec:intro-H2}.
\[
    \begin{tikzcd}
        W_S\ar[rr,rightsquigarrow,"\text{polynomiality}","\text{Sections~\ref{sec:pol D}, \ref{sec:def wt D}}"'] &[-.7cm] & \widetilde{W}_{\mathcal{M}} \ar[r,rightsquigarrow,"\text{Weyl reduction}","\text{Section~\ref{ssec:reduction weyl}}"'] \ar[dl,rightsquigarrow,sloped,"\text{unreduction}","\text{Corollary~\ref{cor:rational}}"']&[.4cm] W_M &[-1.2cm] U(\mathfrak{sl}_2)\ar[l,phantom,"\supset"]\\[-.4cm]
        & \mathcal{H}_2& & & \\
        \bigoplus_{d\geq 0} H^*_\taut(\mathcal{M}_{r,d})\ar[loop, ->, in=70,out=110,looseness=21,shorten=2mm]& H^*_\taut(\mathcal{M}_{r,1})[x]\ar[loop, ->, in=70,out=110,looseness=8,shorten=1mm]&H^*_\taut(\mathcal{M}_{r,1})\ar[loop, ->, in=70,out=110,looseness=21,shorten=2mm]&H^*_\taut(M_{r,1})\ar[loop, ->, in=70,out=110,looseness=21,shorten=2mm]& \\[-.3cm]
        \bigoplus_{d\geq 0} H^*_\taut(M_{r,d})[y]\ar[u,equal]&H^*_\taut(M_{r,1})[x,y]\ar[u,equal]&H^*_\taut(M_{r,1})[y]\ar[u,equal]&&
    \end{tikzcd}
\] 

There is one important (and well known) caveat in the above strategy: punctual modifications of a semistable Higgs bundles need not remain semistable; in other words, Hecke correspondences do not directly act on the stacks of semistables Higgs bundles.

To circumvent this issue, we restrict ourselves to the so-called elliptic locus, where the spectral curve is reduced and irreducible (which implies stability, so that the above problem disappears) and then consider parabolic Higgs bundles (to ensure that the restriction to the elliptic locus is injective in cohomology).
This setup is studied in more detail in Section~\ref{sec:Higgs-bun}.
A special feature of this case is that we have an additional freedom in choosing the $\mathfrak{sl}_2$-triple, which allows us to ensure that $\mathfrak{e}\in \mathfrak{sl}_2$ is the operator of cup product with a relatively ample class, as required (Corollary~\ref{cor:generic sl_2}).
The perverse filtration matches the Chern filtration more or less by construction of this $\mathfrak{sl}_2$ triple (Proposition~\ref{prop:P and sl2}).
We then bootstrap from the elliptic locus to our cases of interest (both usual and parabolic stable Higgs bundles) via a series of \textit{reduction steps}, which are explained in Section~\ref{sec:PW}.
We refer the reader to Section~\ref{sec:PW1} for a summary.
Schematically, we transport the action of $\mathfrak{sl}_2$ by the following maps:
\[
    \begin{tikzcd}
       \mathfrak{sl}_2\ar[r,phantom,"\curvearrowright"] &[-.4cm] H^*_\taut(M^{\parell}_{r,d,D})\ar[r,equal,"\text{Prop.~\ref{prop:elliptic injective}}"] &[.2cm] H^*(M^{\parab}_{r,d,D})\ar[r,equal,"\text{Prop.~\ref{prop:par-to-par0}}"] &[.2cm] H^*(M^{\parab,0}_{r,d,D})\ar[r,shift right=1mm] &[.5cm] H^*(M_{r,d})\ar[l,hook',shift right=1mm,"\text{Thm.~\ref{thm:PW classical}}"'{yshift=2pt}].
    \end{tikzcd}
\]

\subsection{The algebra $\CH_2$.}\label{sec:intro-H2}
We now mention one important source of inspiration and motivation for our construction, which will also explain the title of this paper.
The Lie algebra of polynomial Hamiltonian vector fields on the plane $\CH_2$ has the following description. Its basis is given by the fields $V_{m,n}$ with Hamiltonian $x^m y^n$; explicitly, we have
	\[
	V_{m,n} = n y^{n-1} x^{m} \frac{\partial}{\partial x} - m x^{m-1} y^{n} \frac{\partial}{\partial y},
	\]
	and the Lie bracket is given as follows:
	\[
	[V_{m,n}, V_{m', n'}] = (m'n - m n') V_{m+m'-1, n+n'-1}.
	\]
	It is expected by physicists (Lev Rozansky, private communication) that $\CH_2$ acts on many geometric invariants.
    For instance, in \cite{gorsky2021tautological}, in order to prove a Lefschetz property, certain operators acting on the so-called $y$-ified Khovanov-Rozansky homologies of links were constructed, and it was speculated that these operations satisfy the relations of $\CH_2$. 
	This and some other similarities with the situation of homologies of links gave us the hint that $\CH_2$ may act on $H^*(M)$.	
	It is easy to check however, that $\CH_2$ does not have any non-zero finite-dimensional representation, so there is no hope for it to act directly on $H^*(M)$.
    Nevertheless, as our constructions show, $\CH_2$ \textit{does} naturally act on the infinite-dimensional space $H^*(M)[x,y]$, where $x$ and $y$ are formal variables satisfying some natural commutation relations with $V_{m,n}$.

    \begin{thm*}[Corollary~\ref{cor:rational}, Proposition~\ref{prop:restriction iso} and Remark~\ref{rem:H2-action}]
        The algebra $\CH_2$ acts on $H^*(M)[x,y]$ in such a way that $\psi_n(1)$ corresponds to the operator with Hamiltonian $y^n$.
    \end{thm*}
	
    
	\subsection{Some remarks}
	(1) The fact that the perverse filtration is controlled by Hecke operators is reminiscent of the construction in~\cite{oblomkov2016geometric}. There, the authors consider an action of trigonometric Cherednik algebra on the cohomology of affine Springer fibers by Hecke correspondences, and degenerate it to an action of rational Cherednik algebra by taking the associated graded with respect to the perverse filtration; the latter is inherited from a realization of certain affine Springer fibers as Hitchin fibers. In our situation, the algebra generated by the operators $D_{m,n}$ can be viewed as a global version of the (spherical) trigonometric Cherednik algebra, and the passage to the operators $\widetilde{D}_{m,n}$ in Section \ref{sec:sl2} can be understood as the rational degeneration.
	
	(2) The most natural conceptual framework for the study of the algebras occurring in this paper is that of CoHAs of curves and surfaces, as considered in e.g. \cite{minets2020cohomological, sala2020cohomological, kapranov2019cohomological,davison2022hennecart}. We expect that using the structural results in~\cite{DHKSV}, in particular the existence of a BPS Lie algebra for the full stack of coherent sheaves on symplectic surfaces, together with the computation of the zero-dimensional CoHA in~\cite{MMSV} one should be able to streamline our proof of the $P=W$ conjecture and, for instance, bypass the use of the elliptic locus or of parabolic Higgs bundles.

    (3) Our action of $\liesl_2$ on $H^*(M_{r,d})$ gives rise to an action of the Weyl group $S_2$, which exchanges (by conjugation) the operators of multiplication by tautological classes and the Hecke operators, as expected by mirror symmetry. 

	(4) The Higgs bundles appearing here are sometimes called $\GL_r$-Higgs bundles. We expect that our results, including the algebra action and $P=W$ translate to the $\PGL_r$ case with only slight modifications.
    
	(5) The original motivation for the $P=W$ conjecture came from the observation that the E-polynomials of character varieties and the conjectural mixed Hodge polynomials have an interesting symmetry, which would be geometrically explained by the so-called \emph{curious Poincar\'e duality} or \emph{curious hard Lefschetz theorem}, see \cite{hausel2008mixed}, \cite{hausel2011arithmetic}. For the perverse filtration, the corresponding symmetry is explained by the relative hard Lefschetz theorem, see \cite{cataldo2005hodge}, \cite{cataldo2012topology}, and Theorem \ref{thm:decomposition} here. The curious hard Lefschetz theorem was established in \cite{mellit2019cell} by working on the Betti side. The present proof of $P=W$ now implies another, independent proof of the curious hard Lefschetz theorem (in the usual and the parabolic setups).
	
    \section{Topological prerequisites}\label{sec:topological}
    All the varieties we consider in this paper are defined over $\mathbb{C}$.
By a \emph{space} we mean an algebraic variety or a (classical) Artin stack $X$ of finite type. All our stacks are in fact of the form $\CX=X\times B\BG_m$ where $X$ is an algebraic variety. We work with the cohomology $H^*(X) = H^*(X,\BQ)$, which is a ring, and the Borel-Moore homology $H_*(X)=H_*(X,\BQ)$, which is a module over $H^*(X)$. We refer to \cite[Ch.~2]{CG} for standard properties of $H^*$ and $H_*$. If $X$ is of pure dimension $n$, we denote by $[X]\in H_{2n}(X)$ the fundamental class; it gives rise to a map
\[
H^i(X) \to H_{2n-i}(X),\qquad \alpha \to [X]\cap \alpha.
\]
When $X$ is smooth, this is an isomorphism, and we will sometimes use it implicitly to identify $H^*(X)$ and $H_{2n-*}(X)$. If $X$ is not smooth, consider a resolution of singularities $\pi:\wt X\to X$. By the projection formula, the above map factors as
\begin{equation}\label{Eq:projformula}
H^i(X) \xrightarrow{\pi^*} H^i(\wt X) \cong H_{2n-i}(\wt X) \xrightarrow{\pi_*} H_{2n-i}(X).
\end{equation}

If $X,Y$ are smooth spaces of dimensions $n,m$ respectively, and $f:X\to Y$ is proper, we have the Gysin map in cohomology
\[
H^i(X) \cong H_{2n-i}(X) \xrightarrow{f_*} H_{2n-i}(Y) \cong H^{2m-2n+i}(Y).
\]
We denote the Gysin map by $f_*$ as well.

\subsection{Correspondences}
Let $Z$ be a space of pure dimension $n$, together with maps $\pi_1:Z\to X$ and $\pi_2:Z\to Y$ such that $\pi_2$ is proper. We call this collection of data a \emph{correspondence}, and consider the \emph{induced map}
\[
H^i(X) \to H_{2n-i}(Y),\qquad \alpha \to \pi_{2*}([Z]\cap \pi_1^*\alpha).
\]
If $Y$ is smooth of dimension $m$ this gives a map $H^i(X) \to H^{2m-2n+i}(Y)$. 
The factorization~\eqref{Eq:projformula} implies that $Z$ can be replaced by a resolution of singularities without changing the induced map. 
If $X$ is proper, we can decompose the induced map as follows:
\[
H^i(X) \xrightarrow{\pi_X^*} H^i(X\times Y) \xrightarrow{(\pi_1,\pi_2)_* [Z] \cap} H_{2n-i}(X\times Y) \xrightarrow{\pi_{Y*}} H_{2n-i}(Y).
\]
Notice that since $\pi_2$ is proper, $(\pi_1,\pi_2)$ is also proper. If moreover $X$ and $Y$ are smooth, the class $(\pi_1,\pi_2)_* [Z]$ corresponds to a class in $H^*(X\times Y)=H^*(X)\otimes H^*(Y)$, which using the Poincar\'e duality on $X$ and a choice of basis in $H^*(X), H^*(Y)$ can be identified with a matrix so that the induced map of the correspondence is the action of this matrix.

\subsection{Purity}
Let $X$ be a smooth space. Consider any smooth compactification $\ol X\supset X$. The \emph{pure part} $H_\pure^*(X) \subset H^*(X)$ is the image of the restriction map $H^*(\ol X)\to H^*(X)$. This definition is independent of the choice of a smooth compactification. Indeed, for any two smooth compactifications $\ol X$, $\ol X'$, let $\ol X''$ be a resolution of singularities of the closure of $X$ in $\ol X\times \ol X'$. Then we have a diagram
\[
\begin{tikzcd}
    & X \arrow{d}{\iota} \ar[r,equal] & X \arrow{d}{\pi_2\iota}\\
    \ol X & \ol X'' \ar[l,"\pi_1"'] \arrow {r}{\pi_2} & \ol X'
\end{tikzcd}
\]
with cartesian square, and therefore by base change we have $\iota^* = (\pi_2\iota)^* \pi_{2*}$. Thus
\[
\image (\pi_1\iota)^* \subset \image \iota^* = \image (\pi_2\iota)^* \pi_{2*} \subset \image (\pi_2\iota)^*,
\]
from which the result follows. 
The following claim is standard; we include the proof for completeness.
\begin{prop}\label{prop:pure-corr}
Let $Z \to X \times Y$ be a correspondence between smooth spaces. The induced map in cohomology preserves the pure part. 
In particular, if $X$ is proper then the image of the induced map is contained in $H^*_\pure(Y)$.
\end{prop}

\begin{proof} We may assume that $Z$ is smooth. Hence it is enough to prove that the pure part is preserved under pullbacks and proper pushforwards by maps between smooth spaces. Let $f: Z \to X$ be such a map. Pick a smooth compactification of $X$, a smooth compactification of $Z$, and resolve singularities of the closure of $Z$ in the corresponding product of compactifications to obtain a commutative diagram
\[
\begin{tikzcd}
    Z \arrow{d}{\iota_Z} \arrow{r}{f} & X \arrow{d}{\iota_X}\\
    \ol Z \arrow{r}{\ol f} & \ol X
\end{tikzcd}
\]
where $\iota_Z$, $\iota_X$ are smooth compactifications. Taking pullbacks implies that $f^* H_\pure^*(X) \subset H_\pure^*(Z)$.
Next, assume $f:Z \to Y$ is proper,  and consider a commutative square as above. The canonical map $j:Z\to \ol{Z} \times_{\ol Y} Y$ is both proper and an open embedding, hence it is an isomorphism, and the base change implies $f_* H_\pure^*(Z) \subset H_\pure^*(Y)$.
\end{proof}

Note in particular that the Chern classes of any vector bundle (or a bounded complex of vector bundles) on a smooth space $X$ are pure. 

\subsection{Diagonal class}
Let $X$ be smooth of pure dimension $n$ and let $\ol X\supset X$ be a smooth compactification. The identity map $X\to X$ factors as a composition of proper maps $X\to \ol X \times X \to X$. By~\eqref{Eq:projformula}, we can rewrite the restriction map $H^*(\ol X)\to H^*_\pure(X)$ as the composition
\[
H^*(\ol X) \xrightarrow{\pi_1^*} H^*(\ol X \times X) \xrightarrow{[\Delta]\cap} H^*(\ol X \times X) \xrightarrow{\pi_{2*}} H^*_\pure(X),
\]
where $\Delta\subset \ol X \times X$ is the diagonal. Thus the classes in $H^*_\pure(X)$ of any K\"unneth decomposition of $[\Delta]$ span $H^*_\pure(X)$.

\subsection{Virtual fundamental class}\label{sec:virtual fund class}
Suppose $X$ is smooth of dimension $n$ and let $\CE$ be a complex vector bundle on $X$ of rank $r$. We have the Thom class
\[
\tau_\CE\in H^{2r}(\Tot_\CE, \Tot_\CE\setminus X),
\]
where $\Tot_\CE$ is the total space of $\CE$. If $s:X\to \Tot_\CE$ is a section, by pullback we obtain a class in $H^{2r}(X, X\setminus Z) \cong H_{2n-2r}(Z)$, where $Z=X\cap s(X)$ is the zero set of $s$, which we call the \textit{virtual fundamental class} and denote $[Z]^\vir$. The image of $[Z]^\vir$ under the pushforward $H_*(Z)\to H_*(X)$ is $[X]\cap c_r(\CE)$. Note that the class $[Z]^\vir$ does depend on the realization of $Z$ as the zero set of a vector bundle. However, if $Z$ has pure dimension $n-r$ and the intersection $X\cap s(X)$ is generically transversal over $Z$, we have $[Z]^\vir=[Z]$. 


\section{Relative Lefschetz theory, perverse filtration and Hecke operators}\label{sec:Relative Hard Lefschetz}
In this section we review some notions of the relative Hard Lefschetz theorem of de Cataldo and Migliorini, and describe its behaviour with respect to correspondences. 

\subsection{Lefschetz structures}
Before summarizing the relative Hard Lefschetz theorem from \cite{cataldo2005hodge} we need to develop some algebraic framework.
\begin{defn}
	A \textit{Lefschetz structure} is a finite dimensional vector space $V$ endowed with a linear endomorphism $\omega$ and a finite increasing filtration $P_\bullet V$ such that we have $\omega P_i V\subset P_{i+2} V$ for all $i$, and for all $k\geq 0$ the map
	\[
	\omega^k: P_{-k} V / P_{-k-1} V \to P_k V / P_{k-1} V
	\]
	is an isomorphism. A map of Lefschetz structures is a map of vector spaces compatible with the filtration and commuting with $\omega$.
\end{defn}
Denote by $\Gr V$ the associated graded
\[
\Gr V = \bigoplus_i \Gr_i V, \qquad \Gr_i V = P_i V / P_{i-1} V.
\]
This is a graded vector space on which $\omega$ induces an operator of degree $2$. A map of Lefschetz structures $\varphi:U\to V$ induces a map $\Gr \varphi: \Gr U \to \Gr V$. The following is well-known.
\begin{prop}
	The triple of a vector space $V$, an increasing filtration $P_\bullet V$ and $\omega:V\to V$ satisfying $\omega P_i V\subset P_{i+2} V$ is a Lefschetz structure if and only if there exists an action of $\liesl_2$ on $\Gr V$ for which $\frake = \omega$ and $\frakh$ acts on $\Gr_i V$ as the multiplication by $i$. For any map of Lefschetz structures $\varphi: U\to V$ the induced map $\Gr \varphi$ commutes with $\frakf$. In particular, the $\liesl_2$-action above is unique.\qed
\end{prop}

Lefschetz structures are similar to mixed Hodge structures~\cite{deligne1971theorie} in the following sense:
\begin{prop}
	The category of Lefschetz structures is abelian and the functor $V \to \Gr V$ is an exact faithful functor from the category of Lefschetz structures to the category of finite dimensional representations of $\liesl_2$.
\end{prop}
\begin{proof}
Let $\varphi:U\to V$ be a map of Lefschetz structures. By the snake lemma applied to the short exact sequence
	\[
	0 \to P_i U / P_{i-k} U \to P_{i+1} U / P_{i-k} U \to P_{i+1} U / P_{i} U \to 0,
	\]
	the corresponding sequence for $V$ and the map between them induced by $\varphi$ we obtain a long exact sequence
	\[
	0 \to \kernel \varphi_{P_i / P_{i-k}} \to \kernel \varphi_{P_{i+1} / P_{i-k}} \to \kernel\varphi_{P_{i+1} / P_{i}} \to \cokernel \varphi_{P_i / P_{i-k}}
	\]
	\[
	 \to \cokernel \varphi_{P_{i+1} / P_{i-k}} \to \cokernel\varphi_{P_{i+1} / P_{i}} \to 0.
	\]
	Here we denote by $\varphi_{P_i/P_j}$ the induced map $P_i U/P_j U \to P_i V/P_j V$.
    
\begin{lem}
    The connecting map $\kernel\varphi_{P_{i+1} / P_{i}} \to \cokernel \varphi_{P_i / P_{i-k}}$ is zero. 
\end{lem}
\begin{proof}
    We proceed by induction. Notice that $\bigoplus_{i} \kernel\varphi_{P_{i+1} / P_{i}} = \kernel \Gr\varphi$ is a representation of $\liesl_2$, and therefore is generated over $\omega$ by \emph{primitive} elements, i.e. elements $x\in\kernel\varphi_{P_{-i} / P_{-i-1}}$ for $i\geq 0$ satisfying $\omega^{i+1} x=0$. The image of such $x$ under the connecting map is some  $y\in \cokernel \varphi_{P_{-i-1}/P_{-i-1-k}}$ satisfying $\omega^{i+1} y=0$.

	Let us show that $\omega^{i}$ is injective on $\cokernel \varphi_{P_{-i}/P_{-i-k}}$ for all $i>0$. By the induction assumption we have a short exact sequence
	\[
	0 \to \cokernel \varphi_{P_{-i-1} / P_{-i-k}} \to \cokernel \varphi_{P_{-i} / P_{-i-k}} \to \cokernel\varphi_{P_{-i} / P_{-i-1}} \to 0,
	\]
	and the same short exact sequence for $i$ instead of $-i$. By snake lemma applied to the maps induced by $\omega^i$, we have a short exact sequence
	\[
	0 \to \kernel(\cokernel \varphi_{P_{-i-1} / P_{-i-k}} \to \cokernel \varphi_{P_{i-1} / P_{i-k}})
	\to
	\]
	\[
	\kernel(\cokernel \varphi_{P_{-i} / P_{-i-k}} \to \cokernel \varphi_{P_{i} / P_{i-k}})\to \kernel(\cokernel \varphi_{P_{-i} / P_{-i-1}} \to \cokernel \varphi_{P_{i} / P_{i-1}}).
	\]
	Since $\cokernel \Gr \varphi$ is a representation of $\liesl_2$, the map $\omega^i:\cokernel \varphi_{P_{-i} / P_{-i-1}} \to \cokernel \varphi_{P_{i} / P_{i-1}}$ is an isomorphism, and therefore its kernel vanishes. We know that $\omega^{i+1}$ on $\cokernel \varphi_{P_{-i-1} / P_{-i-k}}$ is injective, since this was already proved when treating the $k-1$ case. Hence $\omega^{i}$ on it is injective, and therefore $\kernel(\cokernel \varphi_{P_{-i-1} / P_{-i-k}} \to \cokernel \varphi_{P_{i-1} / P_{i-k}})$ vanishes, and we conclude that the remaining kernel also has to vanish.

	Coming back to the element $y\in \cokernel \varphi_{P_{-i-1}/P_{-i-1-k}}$, we see that $\omega^{i+1} y=0$ implies $y=0$, and therefore the connecting map vanishes on all primitive elements, and therefore on all elements.
\end{proof}

	Having established the vanishing of connecting map for all $k$, taking $k$ sufficiently large we see that for all $i$ the map
	\[
	\kernel \varphi \cap P_i U \to \kernel \varphi_{P_{i}/P_{i-1}},
	\]
	is surjective, which implies that the natural map
	\[
	\kernel \varphi \cap P_i U / \kernel \varphi \cap P_{i-1} U \to \kernel \varphi_{P_{i}/P_{i-1}}
	\]
	is an isomorphism; equivalently, $\Gr \kernel \varphi \to \kernel \Gr \varphi$ is an isomorphism. Thus $\kernel\varphi$ is a Lefschetz structure. Similarly, we obtain that $\cokernel \Gr \varphi \to \Gr \cokernel \varphi$ is an isomorphism, and so $\cokernel\varphi$ is a Lefschetz structure as well. This implies that the category of Lefschetz structures has kernels and cokernels given by the usual kernels and cokernels together with the induced filtrations. To see that it is abelian, we use the commutative square
	\[
	\begin{tikzcd}
		\cokernel \kernel \Gr \varphi \ar{r}{\sim} \ar[d, "\sim" {anchor=south, rotate=90, inner sep=.5mm}] & \kernel \cokernel \Gr \varphi \\
		\Gr \cokernel \kernel \varphi \ar{r} & \Gr \kernel \cokernel \varphi \ar[u, "\sim" {anchor=south, rotate=90, inner sep=.5mm}]\\
	\end{tikzcd}
	\]
	to see that $\Gr \cokernel \kernel \varphi \to \Gr \kernel \cokernel \varphi$ is an isomorphism. This implies that the two filtrations on $\image\varphi$ coincide. We have shown that the category of Lefschetz structures is abelian. We have also shown that kernels and cokernels commute with $\Gr$, which means that $\Gr$ is exact. Finally, it is faithful because $\Gr V=0$ and finiteness of the filtration implies $V=0$.
\end{proof}

\subsection{Perverse filtration}
Let us recall the definition of the perverse filtration associated to a projective morphism; we refer the reader to~\cite[Section~4.2]{cataldo2005hodge} for the details.
\begin{defn}\label{def:perv}
    Let $f: X\to Y$ be a projective map of algebraic varieties.
    The $i$-th step of the \textit{perverse filtration on $H^*(X)$ associated to $f$} is given by
    \[
        P_i H^{j}(X) \coloneqq \image \left(\mathbb{H}^j(Y,{}^p\tau_{\leq i+n}f_*\mathbb{Q}_X)\to \mathbb{H}^j(Y, f_*\mathbb{Q}_X)\right)\subset \mathbb{H}^j(Y, f_*\mathbb{Q}_X)\simeq H^j(X), 
    \]
    where $\mathbb{H}^*$ denotes sheaf hypercohomology, and ${}^p\tau_{\leq \bullet}$ is the perverse truncation functor.
\end{defn}

We will use some properties of the perverse filtration, as proved by de Cataldo and Migliorini.
Recall that for a nilpotent operator $N$ acting on a vector space, the increasing \emph{canonical filtration} $W_\bullet$ is defined by
\[
W_k = \sum_{i\geq \max(0,k)} \kernel N^{i+1} \cap \image N^{i-k}.
\]

\begin{thm}[See \cite{cataldo2005hodge}, Theorems 2.1.1, 2.1.4, 2.3.3 and Sections 4.5, 4.6]\label{thm:decomposition}
Let $f:X\to Y$ be a map as in Definition~\ref{def:perv}.
Suppose that $X$ is smooth of dimension $n$. Let $\omega\in H^2(X)$ be a relatively ample class, i.e. a class whose positive multiple comes from an embedding $X\to Y\times \BP^N$.
Then the perverse filtration $P_\bullet H^*(X)$ associated to $f$ together with the operator of cup product by $\omega$ form a Lefschetz structure on $H^*(X)$. Moreover, 
\begin{enumerate}
    \item If $U\subset Y$ is open, then the restriction $H^*(X)\to H^*(f^{-1} (U))$ is a map of Lefschetz structures;
    \item If $Y$ is projective and $\omega$ is the pullback of an ample class on $Y$, then $P_i H^*(X) = \bigoplus_{b\in\BZ} W_{i-b} H^{n+b}(X)$, where $W_\bullet$ is the canonical filtration induced by the nilpotent operator of cup product by $\omega$.
\end{enumerate} 
\end{thm}

Let us analyze the effect of pullback and Gysin maps on the perverse filtration. We assume that a quasi-projective base $A\subset \BP^M$ is fixed. Let $\ol A$ be the closure of $A$. Any projective map $\pi:X\to A$ induces an embedding $X\subset \BP^N\times\BP^M$, and resolving the singularities of the closure $\ol A$ we obtain a smooth projective compactification $\iota:X\to \ol X$ together with a map $\ol X\to \BP^N\times \BP^M$. Let $\omega$ be an ample class of $\ol X$. By pullback we obtain a map of Lefschetz structures
\[
\iota^*: H^*(\ol X) \to H^*(X),
\]
which induces a Lefschetz structure on $H^*_\pure(X)$. In particular, we have
\[
P_k H^*_\pure(X) = P_k H^*(X) \cap H^*_\pure(X)  = \iota^* P_k H^*(\ol X),
\]
where the perverse filtration on $H^*(\ol X)$ is associated to the projection $\ol X\to \BP^M$.
Any $\alpha\in H^i(\ol X)$ induces the cup-product map $H^*(\ol X)\to H^{*+i}(\ol X)$, which commutes with $L$, which is the pullback of the hyperplane class from $\BP^M$. So $\alpha$ preserves the canonical filtration associated with $L$ and therefore we have
\[
\alpha P_k H^*(\ol X) \subset P_{k+i} H^{*+i}(\ol X).
\]
Restricting to $H^*(X)$, we obtain the following:
\begin{prop}\label{prop:perverse cup product}
	For any $\alpha\in H^i_\pure(X)$ we have $\alpha P_k H^*_\pure(X)\subset P_{k+i}H^{*+i}_\pure(X)$.\qed
\end{prop}
Next, suppose we have a commutative diagram
\begin{equation}\label{eq:cd1}
\begin{tikzcd}
	X \arrow{r}{f} \arrow{d}{\pi_X} & Y \arrow{ld}{\pi_Y}\\
	A
\end{tikzcd},
\end{equation}
where $\pi_X$ and $\pi_Y$ are projective and $X$, $Y$ are smooth of dimensions $d_X, d_Y$. Note that the map $f$ is automatically projective. We pick smooth compactifications $\iota_X:X\to \ol X$, $\iota_Y:Y\to \ol Y$ and a commutative diagram
\[
\begin{tikzcd}
	\ol X \arrow{r}{\ol f} \arrow{d}{\pi_{\ol X}} & \ol Y \arrow{ld}{\pi_{\ol Y}}\\
	\ol A
\end{tikzcd},
\]
which produces \eqref{eq:cd1} when restricted to $A$. The pullback $H^*(\ol Y)\to H^*(\ol X)$ and the Gysin map $H^*(\ol X) \to H^{*+2d_Y-2d_X}(\ol Y)$ commute with the multiplication by $L$ and therefore preserve the canonical filtration associated to it. So we obtain
\begin{prop}\label{prop:functoriality}
	For a commutative diagram of the form \eqref{eq:cd1} with $\pi_X$, $\pi_Y$ projective and $X$, $Y$ smooth of dimensions $d_X, d_Y$ the pullback map and the Gysin map satisfy
	\[
	f^* P_k H^*_\pure (Y) \subset P_{k+d_Y-d_X} H^*_\pure(X),\qquad f_* P_k H^*_\pure (X) \subset P_{k+d_Y-d_X} H^{*+2 d_Y - 2 d_X}_\pure(Y).
	\]
\end{prop}

More generally, suppose that we have a commutative diagram
\begin{equation}\label{eq:cd2}
	\begin{tikzcd}
		X \arrow[swap]{dr}{\pi_X} & Z \arrow{d}{\pi_Z} \arrow[swap]{l}{\pi_1} \arrow{r}{\pi_2} & Y \arrow{dl}{\pi_Y}\\
		& A &
	\end{tikzcd},
\end{equation}
where $X$, $Y$, $Z$ are smooth and $\pi_X$, $\pi_Y$, $\pi_Z$ are projective. Let us also assume that $Z$ carries a virtual fundamental class $[Z]^\vir \in H_*(Z)$, see Section~\ref{sec:virtual fund class}. Consider the action of the correspondence
\begin{equation}\label{eq:correspondence 2}
\alpha \to \pi_{2*}(\pi_1^* \alpha \cap [Z]^\vir).
\end{equation}
By~\eqref{Eq:projformula}, we can assume that $Z$ is smooth.
Suppose the dimensions of $X$, $Y$, $Z$ are $d_X$, $d_Y$, $d_Z$ respectively. Then Proposition~\ref{prop:functoriality} implies
\begin{prop}
	The correspondence~\eqref{eq:correspondence 2} sends $P_i H^*(X)$ to $P_{i+d_X+d_Y - 2 d_Z + 2k} H^{*+2 d_Y - 2 d_Z + 2 k}(Y)$.
\end{prop}
Notice that $d_Z$ and $k$ enter the statement via the difference $d_Z-k$, which corresponds to the difference between the (naive) dimension and the virtual dimension of $Z$.

\begin{rem}
(i) When $2 d_Z - 2k = d_X + d_Y$, e.g. for $k=0$, $X$ and $Y$ symplectic, and $Z$ a Lagrangian in $X\times Y$, the correspondence preserves the perverse filtration.\\
(ii) When $d_X=d_Y$, a correspondence of cohomological degree $j$ sends $P_i$ to $P_{i+j}$.
\end{rem}

\section{Hecke correspondences}\label{sec:hecke}
\subsection{Notations}\label{ssec:hecke notations}
Let $S$ be a smooth projective surface, and set $H=H^*=H^*(S,\mathbb{Q})$. Let $\delta\in H^4$ be the class of a point; by abuse of notation we write $\Delta\in \bigoplus_{i} H^i\otimes H^{4-i}$ for the class of the diagonal. Denote the Chern roots of the tangent bundle of $S$ by $t_1, t_2$. We then have
\[
\Td^{-1}_S = \frac{(1-e^{-t_1})(1-e^{-t_2})}{t_1 t_2} = 1 - \frac{t_1+t_2}{2} + \frac{2t_1^2 + 2t_2^2 -3t_1t_2}{12}.
\]

In this section we consider actions of correspondences on moduli stacks of coherent sheaves on $S$ induced by pointwise modifications. Such actions are best studied using Cohomological Hall algebras, as in e.g. \cite{MMSV}, \cite{negut2019shuffle}. However we will not need the full generality of \textit{loc.cit.} and thus only cite the relevant results. In particular, as we will consider correspondences over \textit{smooth} stacks only, we will freely pass between cohomology and Borel-Moore homology. 

Fix a non-zero $\alpha_0\in H^\ev \coloneqq H^0 \oplus H^2 \oplus H^4$. For any $\alpha=\alpha_0 + n\delta$ with $n \in \BN$ we fix a smooth open substack $\CM_\alpha$ of the (classical) moduli stack of coherent sheaves on $S$ of Chern character $\alpha$.
We denote by $\Coh_\delta$ the stack of length $1$ sheaves on $S$; it is explicitly given by $\Coh_\delta=S\times B\BG_m$. We denote by $\CF_\alpha$ the universal coherent sheaf on $\CM_\alpha\times S$, and similarly by $\CF_\delta$ the universal coherent sheaf on $\Coh_\delta\times S$.

We make the following two running assumptions:
\begin{enumerate}[label=(A\arabic*),ref=A\arabic*]
\item \label{assumption-a} For any $\alpha$ as above, $\CM_\alpha$ parametrizes coherent sheaves on $S$ with no finite length subsheaves,
\item \label{assumption-b} For any $n>0$, a coherent sheaf $\CF \in \CM_{\alpha_0+n\delta}$ and a colength one subsheaf $\CG \subset \CF$ we have $\CG \in \CM_{\alpha_0+(n-1)\delta}$.
\end{enumerate}

Such assumptions are verified in several examples of interest, e.g. for the elliptic locus of stable Higgs bundles, see Section~\ref{sec:Higgs-bun}. 

\subsection{Length one Hecke correspondences}\label{ssec:hecke hecke}
By assumption~\eqref{assumption-a}, the universal sheaf $\CF_\alpha$ over $\CM_\alpha$ admits a length two resolution $\CE_1\to\CE_0 \to \CF_\alpha$
by vector bundles. We may thus consider the projectivization  
$\CZ_\alpha\coloneqq\BP(\CF_\alpha)$, which may be explicitly realized as the zero set of the canonical section of the vector bundle $\CE_1^*(1)$ on the projective bundle $\BP(\CE_0)$. The stack $\CZ_\alpha$ parameterizes colength $1$ subsheaves of $\CF_\alpha$, c.f.~\cite[Sec. 8]{Jiang}, or equivalently flags $F'\subset F$ of sheaves on $S$ with $\ch F = \alpha$, and quotient $F/F'$ of length $1$.

By assumption~\eqref{assumption-b}, we have a natural map $\CZ_\alpha\to \CM_{\alpha-\delta}$.
Consider the following \textit{Hecke correspondence}:
\begin{equation}\label{diag:Hecke corr}
\begin{tikzcd}
    & \CZ_\alpha \arrow[dl,"{(\pi_1,\pi_2)}"'] \arrow{rd}{\pi_3} & \\
    \Coh_\delta \times \CM_{\alpha-\delta}& & \CM_\alpha
\end{tikzcd}\qquad
\begin{aligned}
    (\pi_1&{,} \pi_2): (F'\subset F)\mapsto (F/F',F'),\\
    \pi_3&: (F'\subset F) \mapsto F.
\end{aligned}
\end{equation}
The map $\pi_3$ is proper by construction.
Therefore, as explained in Section~\ref{sec:topological}, we obtain the following map, which we denote by $T$:
\begin{equation*}
\begin{split}
H^*(\Coh_\delta) \otimes H^*(\CM_{\alpha-\delta}) &\to H^*(\CM_\alpha),\\ \xi\otimes c &\mapsto T(\xi)(c) = \pi_{3*}([\CZ_\alpha]^\vir \cap \pi_1^*(\xi) \cap \pi_2^*(c)).
\end{split}
\end{equation*}
For each $\xi \in H^*(\Coh_\delta)$, we will call the induced map $T(\xi) : H^*(\CM_{\alpha-\delta}) \to H^*(\CM_\alpha)$ a \textit{Hecke operator}.

\vspace{.1in}

In the rest of this Section we explicitly compute the action of the above Hecke operators on the tautological ring $H_{\taut}^*(\CM_{\alpha-\delta}) \subseteq H^*(\CM_{\alpha-\delta})$, whose definition we first recall below.

\subsection{The tautological classes}\label{subs:taut-classes} 
Denote by $\Lambda$ the Macdonald ring of symmetric functions in infinitely many variables.
We have $\Lambda=\BC[p_1,p_2,\ldots]$, where $p_k$ is the power sum function of degree $k$.
For each $\alpha$ there is a ring homomorphism $\Lambda\to H^*(\CM_\alpha\times S)$, which sends $p_k$ to the element $p_k(\CF_\alpha)\in H^{2k}(\CM_\alpha\times S)$ defined via the generating series:
\begin{equation}\label{eq:pn-Fal}
    \ch(\CF_\alpha) = \rank \CF_\alpha + \sum_{n=1}^\infty \frac{p_n(\CF_\alpha)}{n!}.
\end{equation}
It is convenient to denote $p_0(\CF_\alpha)=\rank \CF_\alpha$. For any symmetric function $f$ of degree $k$ and any $\gamma\in H^m$ we set
\begin{equation}\label{eq:taut-eval}
    f(\gamma) = \pi_{\CM_\alpha *} \left(f(\CF_\alpha) \cup \pi_{S}^* \gamma\right) \in H^{2k+m-4}(\CM_\alpha),
\end{equation}
where $\pi_{\CM_\alpha}:\CM_\alpha\times S\to \CM_\alpha$, $\pi_{S}:\CM_\alpha\times S\to S$ are the projections.
The classes $f(\gamma)$ are pure by Proposition~\ref{prop:pure-corr}, and so generate a subring $H^*_\taut(\CM_\alpha)\subset H^*_\pure(\CM_\alpha)$, which we call the \emph{tautological ring}. 

Let us write $\Delta\gamma\coloneqq\Delta \cup(\gamma\otimes 1) = \Delta \cup(1\otimes \gamma)$.
The product of symmetric functions is related to the product in $H^*_\taut(\CM_\alpha)$ by
\begin{equation}\label{eq:prod-of-sym-fts}
    (fg)(\gamma) = (f\otimes g)(\Delta\gamma),
\end{equation}
where the right hand side is defined via $(f\otimes g)(\gamma_1\otimes\gamma_2) = f(\gamma_1) g(\gamma_2)$.

Let us now replace $\alpha$ with $\delta$. 
Recall the tautological sheaf $\CF_\delta$ on $\Coh_\delta\times S$.
We have $\CF_\delta = \CO_\Delta(1)$, where $\CO_\Delta$ is the structure sheaf of the diagonal in $S\times S$ and $(1)$ means tensoring with the weight $1$ character of $\mathbb{G}_m$. The Chern character of $\CO_\Delta$ is $\Delta \Td^{-1}_S$, therefore
\[
\ch(\CF_\delta) = \ch(\CO_\Delta)\ch(\CO(1)) = \Delta \Td^{-1}_S e^u,
\]
where $u=c_1(\CO(1))$ is the equivariant parameter.
Define $p_n(\CF_\delta)$ as in~\eqref{eq:pn-Fal}, and for any $\eta \in H$ consider the evaluation $p_n(\CF_\delta)(\eta)\coloneqq \pi_{\Coh_\delta *} \left(p_n(\CF_\delta) \cup \pi_{S}^* \eta\right)$ as in~\eqref{eq:taut-eval}.
We deduce that
\[
p_n(\CF_\delta)(\eta) = \frac{u^n - (u-t_1)^n - (u-t_2)^n + (u-t_1-t_2)^n}{t_1 t_2} \eta.
\]

\subsection{The action of the Hecke operators on tautological classes}\label{subs:act-on-taut}
In order to give the explicit expression of length one Hecke correspondences on tautological classes, we need to introduce a few combinatorial definition. Consider the super-commutative ring
\begin{equation}\label{eq:def Lambda S}
\Lambda_S\coloneqq \BQ[p_k(\eta)\;;\; k \geq 0, \eta \in H]/I
\end{equation}
where $\deg(p_k(\eta))=2k-4+\deg\eta$ and
the ideal $I$ is generated by the following relations: 
\begin{equation}\label{eq:taut-ring-rels}
\begin{aligned}
    p_k&(\eta+\lambda)=p_k(\eta)+p_k(\lambda), \quad p_k(a\eta)=ap_k(\eta),\\
    p_0&([\mathrm{pt}])=\rank \CF_\alpha, \qquad p_k(\lambda) =0 \;\text{if}\; 2k+\deg \lambda <4
\end{aligned}
\end{equation}
for any $k \geq 0$, $\eta,\lambda \in H$ and $a\in \BQ$.
Note that we treat the elements $p_k(\eta)$ as formal variables in this definition.
However, for any $\alpha$ there is a canonical evaluation morphism $\Lambda_S \to H^*(\CM_\alpha)$ sending $p_k(\lambda) \in \Lambda_S$ to $p_k(\lambda)\in H^*(\CM_\alpha)$ as defined in~\eqref{eq:taut-eval}.
By a slight abuse of notation, we will use $p_k(\lambda)$ to denote either element.
In this way, the ring $\Lambda_S$ may be interpreted as the universal ring of tautological classes.

Let us set $\V_S=\bigoplus_{\alpha \in \alpha_0+\BN\delta} H^*(\CM_\alpha)$ and consider the linear map
\begin{equation*}
\begin{split}
\Phi\colon\Lambda_S\otimes H[u] &\to \End(\V_S),\\
f\otimes \xi u^n &\to f T(\xi u^n)
\end{split}
\end{equation*}
for $f\in\Lambda_S,\;\xi\in H,\; n\geq 0$. Note that $\Phi$ is \textit{not} a ring homomorphism, but a morphism of $\Lambda_S$-modules for the action of $\Lambda_S$ on $\Lambda_S\otimes H[u]$ by left multiplication. Next, let us define a ring homomorphism
\begin{equation*}
\begin{split}
R\colon\Lambda_S &\to \Lambda_S\otimes H[u],\\
p_n(\eta) &\to p_n(\eta)-\frac{u^n - (u-t_1)^n - (u-t_2)^n + (u-t_1-t_2)^n}{t_1 t_2} \eta,
\end{split}
\end{equation*}
and a $\Lambda_S$-linear map
$$Q: \Lambda_S \otimes H[u] \to \Lambda_S, \qquad f\otimes \xi u^n \mapsto f h_n(\xi).$$

\begin{prop}\label{prop: (4.9)}
The subspace
\begin{equation}\label{eq:Vtaut-def}
    \V_S^{\taut}\coloneqq\bigoplus_{\alpha \in \alpha_0+\BN\delta} H^*_{\taut}(\CM_\alpha);   
\end{equation}
is preserved by the action of Hecke operators and for any $f \in \Lambda_S, \xi \in H$ and $n\geq 0$ we have
\begin{equation}\label{eq:T from Q R}
T(\xi u^n)(f\cdot 1)=\Phi(\xi u^n\cdot R(f))(1)=Q(\xi u^{n-\rank\;\CF_\alpha+1} R(f))
\end{equation}
\end{prop}
This Proposition is proved in \cite[Sec.~2.6]{MMSV} in a more general context. However, for the comfort of the reader, we sketch the proof below.

\subsection{A sketch of proof of Proposition~\ref{prop: (4.9)}}
Consider the space $\CZ_\alpha\times S$ equipped with the maps $\pi_{i4}=\pi_i\times\pi_{S} $, ($i=1,2,3$).
The projection formula together with the identity $[\pi_{14}^* \CF_\delta] + [\pi_{24}^* \CF_{\alpha-\delta}] = [\pi_{34}^* \CF_{\alpha}]$ in $K_0(\CZ_\alpha\times S)$ imply the following identity of operators $H^*(\CM_{\alpha-\delta}) \to H^*(\CM_\alpha)$:
\begin{equation}\label{eq:hecke commutation}
    [p_n(\eta), T(\xi)] = T\left(\frac{u^n - (u-t_1)^n - (u-t_2)^n + (u-t_1-t_2)^n}{t_1 t_2} \eta\xi \right)
\end{equation}
for any $\eta \in H$, $\xi \in H^*(\Coh_\delta)$ and $n \geq 0$.
The commutator on the left hand side is \textit{graded}, i.e. if both $\eta$ and $\xi$ are odd we take the sum instead of the difference.

Thanks to the commutation relations~\eqref{eq:hecke commutation}, in order to fully determine the action of the Hecke operators on the tautological ring $H^*_\taut(\CM_\alpha)$ we only need to compute the action of each Hecke operator on $1$.

The map $\CZ_\alpha\to \Coh_\delta = S \times B\BG_m$ has two components. The component corresponding to $S$ comes from the projection $\BP(\CE_0)\to S$. The component corresponding to $B\BG_m$ is given by the tautological line bundle $L$ on $\BP(\CE_0)$. The class $c_1(L)$ thus gets identified with the pullback of $u$. We will denote $c_1(L)=u$ by abuse of notation.

Recall that $\CZ_\alpha$ is presented as the zero set of a section of the vector bundle $\CE_1^*(1)$ on $\BP(\CE_0)$. The projection $\pi_3$ factors as the composition of proper maps
\[
\CZ_\alpha \to \BP(\CE_0) \to \CM_\alpha \times S \to \CM_\alpha.
\]
 We have
\[
H^*(\BP(\CE_0)) = H^*(\CM_\alpha \times S)[u]/\left(u^m - c_1(\CE_0) u^{m-1} + \cdots + (-1)^m c_m(\CE_0)\right),
\]
where $m$ is the rank of $\CE_0$.
The Gysin map $H^*(\BP(\CE_0))\to H^*(\CM_\alpha \times S)$ applied to a polynomial $p(u)$ can be computed as follows: expand
\[
\frac{p(u)}{u^m - c_1(\CE_0) u^{m-1} + \cdots \pm c_m(\CE_0)}
\]
as a power series in $u^{-1}$, and take the coefficient of $u^{-1}$. On the other hand, the Euler class of the bundle $\CE_1^*(1)$ is given by the expression
\[
u^{m'} - c_1(\CE_1) u^{m'-1} + \cdots + (-1)^m c_{m'}(\CE_1),
\]
where $m'$ is the rank of $\CE_1$.
So the result of taking an element $u^n$, multiplying it by the Euler class of $\CE_1^*(1)$ and applying the Gysin map to $H^*(\CM_\alpha \times S)$ is the coefficient of $u^{-n-1}$ in the expansion of
\[
\frac{u^{m'} - c_1(\CE_1) u^{m'-1} + \cdots \pm c_{m'}(\CE_1)}{u^m - c_1(\CE_0) u^{m-1} + \cdots \pm c_m(\CE_0)} = \sum_i u^{-\rank \CF_\alpha-i} s_i(\CF_\alpha),
\]
which is the generating series of the Segre classes of $\CF_\alpha = \CE_0/\CE_1$.
In our notation, the Segre classes correspond to the complete homogeneous symmetric functions $h_n$. Hence we obtain
\begin{equation}\label{eq:hecke segre}
    T(\xi u^n)(1) = h_{n+1-\rank\CF_\alpha}(\xi),\qquad \xi\in H,\; n\geq 0.
\end{equation}

Formula~\eqref{eq:T from Q R} now follows from the combination of \eqref{eq:hecke commutation} and \eqref{eq:hecke segre}.

\subsection{The case of an open surface}\label{sec:open surface}
So far, we have only considered the case of a projective surface $S$. All of the results above naturally extend to the quasiprojective case, see \cite{MMSV}. More precisely, let $S_0\subsetneq S$ be an open subset and assume that $\CM_\alpha$ parameterizes sheaves with (proper) support in $S_0$. This automatically implies that such sheaves are of rank zero.
In this situation, one can still define Hecke correspondences
$$H^*(\Coh(S_0)_\delta) \otimes H^*(\CF_\alpha) \to H^*(\CF_{\alpha+\delta})$$ 
via the diagram \eqref{diag:Hecke corr}, provided that the assumptions~\eqref{assumption-a}, \eqref{assumption-b} hold. The tautological class $f(\gamma) \in H^*(\CM_\alpha)$ only depends on the restriction of $\gamma$ to $H^*_{\pure}(S_0)$; indeed, 
if $\gamma\in\kernel (H^*(S) \to H^*(S_0))$ then the long exact sequence in Borel-Moore homology shows that $\pi_{S}^* \gamma$ comes from $H_*(\CM_\alpha\times (S\setminus S_0))$, on which the cap product by $\ch \CF_\alpha$ vanishes. This allows us to consider a ring of tautological classes $\Lambda_{S_0}$ defined as in \eqref{eq:def Lambda S} by replacing $H^*(S)$ with $H^*_{\pure}(S_0)$, as well as a family of Hecke operators $T(\xi u^n)$ for $\xi \in H^*_{\pure}(S_0)$ and $n\geq 0$ acting on $\V_S$. In that setting, the formula~\eqref{eq:T from Q R} remains valid by an identical argument.

\begin{rem}
    One caveat to keep in mind is that the value $h_0(\gamma)=1(\gamma)$ is undefined, but because $\rank\CF_\alpha=0$, the argument of $Q$ in \eqref{eq:T from Q R} has only strictly positive powers of $u$, so that $h_0(\gamma)$ never occurs.
\end{rem}

\section{W-algebras and the Fock space representation}\label{sec:fock space}
In this short section we recall the presentation, found in \cite{MMSV}, of the algebra generated by the Hecke operators $T(\xi u^n)$ acting on $\V_S$. Throughout, we fix 
a smooth quasiprojective surface $S$, set $H=H^*_{\pure}(S)$ and denote by $c_1,c_2$ the Chern classes of the tangent bundle of $S$.

\subsection{The Fock space}\label{sec:Fock general}
We keep the notations from the previous section. In particular, we consider the graded super-commutative algebra $\Lambda_S=\bigoplus_{n \geq 0} \Lambda_S^n$ defined by~\eqref{eq:def Lambda S}. We define the Fock space as
$$\bfF_S\coloneqq \Lambda_S[s] \simeq  \bigoplus_{n\geq 0} \Lambda_S s^n.$$
This is a $\BN^2$-graded vector space with $\bfF_S[a,b]=\Lambda_S^a s^b$.
We call the first, resp. second grading the \textit{vertical}, resp. \textit{horizontal} grading.



Define, for each $\xi\in H$ and $n\geq 0$, an element $\psi_n(\xi)\in \Lambda_S$ by the relation
\begin{equation}\label{eq:psi definition}
    \psi_n(\xi) = \sum_{1\leq m\leq n+1} \frac{n!}{m!} p_m(\Td_{n+1-m} \xi); \qquad \sum_{k\geq 0} \Td_{k} x^k \coloneqq \frac{t_1 t_2 x^2}{(1-e^{-t_1 x})(1-e^{-t_2 x})}.
\end{equation}
In other words, the $\psi_n$'s correspond to the coefficients of $\Td_S \ch(\CF_\alpha)$, and the generating series of $\psi$ is
\[
\sum_{n\geq 0} \frac{x^n \psi_n}{n!} = \left(\sum_{n\geq 1} \frac{x^n p_n}{n!}\right) \frac{t_1 t_2 x}{(1-e^{-t_1 x})(1-e^{-t_2 x})}.
\]
We regard $\psi_n(\xi)$ as an operator on $\Lambda_S$ by left multiplication.

\begin{rem}
Observe that any $p_n$ for $n\geq 1$ is a linear combination of $\psi_m$'s. The remaining element $p_0$ will not play a role in this work (but see \cite[Sec.~1.6]{MMSV} for its geometric meaning).    
\end{rem}

We regard the operators $T_n(\xi), \psi_m(\xi)$ as acting on $\bfF_S$ in the following way:
$$T_n(\xi): \bfF_S[a,b] \to \bfF_S[a+2n+\xi,b+1],$$
$$\psi_m(\xi):\bfF_S[a,b] \to \bfF_S[a+2m-2+\xi,b],$$
i.e. $T_n(\xi), \psi_m(\xi)$ are of respective horizontal degrees $1$ and $0$.

\begin{thm}\label{thm:surface W relations} Set $s_2=c_1^2 - c_2$. The action of the operators $\psi_m(\xi), T_n(\xi)$, $m,n\geq 0$, $\xi\in H$ on $\bfF_S$ satisfies the following relations:
    \begin{equation}\label{eq:Q -1}
    [\psi_m(\xi), \psi_n(\eta)] = 0,\quad \psi_m(a\xi + b\eta) = a\psi_m(\xi) + b\psi_m(\eta),
    \end{equation}
    \begin{equation}\label{eq:Q0}
        [\psi_m(\eta), T_n(\xi)] = m T_{m+n-1}(\eta\xi),
    \end{equation}
    \begin{equation}\label{eq:Q1}
        [T_m(\xi\xi'), T_n(\xi'')] = [T_m(\xi), T_n(\xi'\xi'')],
    \end{equation}
    \begin{multline}\label{eq:Q2}
        [T_{m}(\xi), T_{n+3}(\xi')] - 3 [T_{m+1}(\xi), T_{n+2}(\xi')] + 3 [T_{m+2}(\xi), T_{n+1}(\xi')] - [T_{m+3}(\xi), T_{n}(\xi')]\\
        - [T_{m}(\xi), T_{n+1}(s_2 \xi')] + [T_{m+1}(\xi), T_{n}(s_2 \xi')] + \{T_m, T_n\}(c_1\Delta\xi \xi')=0,
    \end{multline}
    \begin{equation}\label{eq:Q3}
        \sum_{\pi\in S_3} \pi [T_{m_3}(\xi_3), [T_{m_2}(\xi_2), T_{m_1+1}(\xi_1)]] = 0.
    \end{equation}
\end{thm}
\begin{proof} For $S$ projective, this results from a direct computation which is performed in \cite[Sec.~3]{MMSV}. The case of a quasiprojective surface $S$ may be deduced by fixing a compactification $\overline{S}$ of $S$, realizing $H^*_{\pure}(S)$ as a quotient $H^*(\overline{S})/I$, where $I$ is an ideal, and using Theorem~\ref{thm:surface W relations} for $\overline{S}$ (see \cite[Sec.~5]{MMSV} for details).
\end{proof}

\begin{rem}
    All the commutators above a super-commutators.
    In \eqref{eq:Q2}, $\{T_m, T_n\}(c_1\Delta\xi \xi')$ stands for the super-anticommutator of the operators $T_m$, $T_n$ whose arguments are taken from the tensor $s_1\Delta\xi \xi'\in H\otimes H$. In \eqref{eq:Q3}, $\pi$ permutes both the indices $m_i$ and elements $\xi_i$. Moreover, when we permute $\xi_i$ the sign changes according to the usual rules depending on the parity of $\xi_i$.
\end{rem}

\subsection{Surface $W$-algebra}\label{sec:undeformed}
Theorem~\ref{thm:surface W relations} motivates the following definition:
\begin{defn}
    The \emph{surface $W$-algebra}\footnote{There is a notational clash with \cite{MMSV}, where $W_S$ is denoted $W^+_S$, as it corresponds to \textit{positive} punctual modifications.} $W_S$ is the algebra generated by elements $\psi_m(\pi)$, $T_m(\pi)$, where $m\geq 0$ and $\pi$ runs over a basis of $H$, modulo relations
    \eqref{eq:Q -1}--\eqref{eq:Q3}.
\end{defn}
Note that $W_S$ is naturally $\mathbb{N}$-graded, where we put $\psi_m(\xi)$ and $T_n(\xi)$ in degrees $0$ and $1$ respectively. By construction, there is a canonical evaluation morphism $\mathrm{ev}\colon\bfF_S \to \V_S$. From the results of Section~\ref{sec:hecke}, we deduce the following. 
\begin{cor}\label{cor:surface W} The canonical evaluation morphism $\text{ev}\colon\bfF_S \to \V_S$, together with the assignment
$$\psi_n(\xi) \mapsto \psi_n(\xi), \qquad T_m(\xi) \mapsto T(\xi u^{m-1+\rank\;\CF_\alpha})$$
intertwines the action of $W_S$ on $\bfF_S$ and of the corresponding tautological and Hecke operators on $\V_S$.
\end{cor}
As the relations \eqref{eq:Q -1}--\eqref{eq:Q3} are invariant under index shifts in the generators $T_m$, the assignment
$$\psi_n(\xi) \mapsto \psi_n(\xi), \qquad T_m(\xi) \mapsto T(\xi u^{m})$$
still yields a representation of $W_S$ on $\V_S$. From now on, we consider this renormalized action. Beware that with this renormalization, formula \eqref{eq:T_n action} only holds up to an index shift of $(1-\rank \;\CF_\alpha)$.


The structure of $W_S$ is rather well understood. Let us consider first the simpler case:
\begin{thm}[{\cite[Thm.~3.5]{MMSV}}] \label{thm:W undeformed}
Assume that both $s_2\in H^4(S)$ and $s_1\Delta\in H^6(S\times S)$ vanish. Then $W_S$ is isomorphic to the universal enveloping algebra of the Lie algebra with basis $D_{m,n}(\pi)$, where $m,n\in\BZ_{\geq 0}$ and $\pi$ runs over a basis of $H^*(S)$, and Lie bracket given by
    \begin{equation}\label{eq:W relations}
        [D_{m,n}(\xi), D_{m',n'}(\eta)] = (n m' - m n') D_{m+m',n+n'-1}(\xi\eta).
    \end{equation}
    The generators are related by $D_{0,n}(\xi) = \psi_{n}(\xi)$, $D_{1,n}(\xi) = T_n(\xi)$.
More generally, for any $\xi \in H$ and $n,m \geq 0$ we have
 \[\label{eq:def D m n undeformed}
    D_{m,n}(\xi) = \frac{n!}{(m+n)!}(-\Ad_{T_0(1)})^m \psi_{m+n}(\xi).
    \]
\end{thm}

\begin{rem}\label{rem:undeformed}
    The assumptions of Theorem~\ref{thm:W undeformed} are satisfied in the following important cases:
    \begin{itemize}
        \item $S$ is an abelian surface,
        \item $S$ is a K3 surface without a point,
        \item $S$ is the total space of a line bundle over a curve.
    \end{itemize}
\end{rem}

The same Lie-theoretic description essentially works for an arbitrary $S$, but only after passing to the associated graded with respect to a suitable filtration. We write $W_S =\bigoplus_m W_S[m]$ for the decomposition into graded pieces for the horizontal grading. Let $F_\bullet$ be the smallest filtration on $W_S$ such that for all $m,n\in\BZ$ and $\xi\in H$ we have
$$\psi_n(\xi), T_n(\xi) \in F_n,\qquad F_m F_n \subset F_{m+n}, \qquad [F_m, F_n] \subset F_{m+n-1}.$$
\begin{thm}[{\cite[Sec. 3]{MMSV}}]\label{thm:deformed W}
There exist elements $D_{m,n}(\xi)\in W_S[m]$, $m,n\geq 0, \xi\in H$ such that
    \begin{enumerate}
        \item $D_{0,n}(\xi) = \psi_n(\xi)$, $D_{1,n}(\xi)=T_n(\xi)$;
        \item $F_{-1}=0$ and $F_N$ is spanned by products $D_{m_1, n_1}(\xi_1)\cdots D_{m_1, n_k}(\xi_k)$ satisfying $n_1+\cdots+n_k\leq N$;
        \item $[D_{m,n}(\xi), D_{m',n'}(\eta)] = (n m' - m n') D_{m+m',n+n'-1}(\xi\eta)$ modulo $F_{n+n'-3}$.
    \end{enumerate}
\end{thm}

The following elements will play an important role in the next section:
$$q_m(\xi) \coloneqq D_{m,0}(\xi), \qquad L_m(\xi) \coloneqq D_{m,1}(\xi), \qquad  \frakd\coloneqq \frac12 \psi_2(1); \qquad m\geq 0, \xi\in H.$$ 
Although we will not be using this, we quote for completeness a simple corollary of Theorem~\ref{thm:W undeformed}, which explains the link of $W_S$ with (a half of) the Virasoro algebra and the element $W_3(0)$ of the Zamolodchikov algebra.

\begin{cor}\label{cor:q L}
The elements $q_m(\xi)$, $L_m(\xi)$, $m\geq 0$, $\xi\in H$ and $\frakd=\frac12 \psi_2(1)$ satisfy relations
\[
[q_m(\xi), q_n(\eta)] = 0,\quad [L_m(\xi), q_n(\eta)] = n q_{m+n}(\xi\eta), \quad [L_m(\xi), L_n(\eta)] = (n-m) L_{m+n}(\xi\eta),
\]
\[
[\frakd, q_m(\xi)] = m L_m(\xi).
\]
\end{cor}

\subsection{Variations}\label{rem:any-H} The construction of $W_S$ and its Fock space representation are purely algebraic. It is useful to slightly extend the setup in which one may consider such objects. In particular, assume given the following data:
\begin{itemize}
    \item A graded super-commutative ring $H = H^*(S)$ over $\BQ$;
    \item An augmentation map $\veps:H\to \BQ$, which is a $\BQ$-linear map of degree $-4$;
    \item An element $\Delta\in H\otimes_\BQ H$ of degree $4$ satisfying $(\veps\otimes\Id)(\Delta) = (\Id\otimes\veps)(\Delta) = 1$ and $(\xi\otimes 1) \Delta = (1\otimes \xi) \Delta\eqqcolon \Delta \xi$ for all $\xi\in H$;
    \item Fixed elements $c_1\in H^2$, $c_2\in H^4$, satisfying $\Delta^2 = c_2\Delta$. We formally write $c_1=t_1+t_2$ and $c_2=t_1 t_2$ and identify symmetric functions in $t_1, t_2$ with polynomials in $c_1, c_2$.
    Geometrically, $t_1, t_2$ are the Chern roots of the tangent bundle of $S$.
\end{itemize}

Just as before, we may define a graded super-commutative algebra $\Lambda_H=\bigoplus_{n \geq 0} \Lambda_H^n$ by~\eqref{eq:def Lambda S} and a Fock space
$$\bfF_H\coloneqq \Lambda_H[s] \simeq  \bigoplus_{n\geq 0} \Lambda_H s^n.$$
For each $\xi\in H$, $n\in \BN$ we define $T_n(\xi):\Lambda_H\to\Lambda_H$ as in Section~\ref{subs:act-on-taut}:
\begin{equation}\label{eq:T_n action}
T_n(\xi)(f) = Q\left(\xi u^n R(f)\right).
\end{equation}
Here, the notation $f(\xi)$ is extended from $f=p_k$ to any symmetric function $f$ by formula~\eqref{eq:prod-of-sym-fts}. The definition of the algebra $W_H=W(H,\varepsilon, \Delta, c_1,c_2)$ as well as the proofs of Theorems~\ref{thm:surface W relations},~\ref{thm:W undeformed} and~\ref{thm:deformed W} from~\cite{MMSV} are entirely algebraic and can be repeated \textit{verbatim} in this context. 
    
Furthermore, let us assume we are given
\begin{itemize}
\item A graded ideal $I \subset H$;
\end{itemize}
and consider the quotient $H' \coloneqq H/I$ instead of $H$.
Then, as indicated in the proof of Theorem~\ref{thm:surface W relations}, the algebra $W_{H'}$ admits a quotient with the same presentation, where the algebra $H$ replaced by $H'$.

\section{Degeneration and the $\liesl_2$}\label{sec:sl2}
We keep the geometric setup of Section~\ref{sec:hecke}; namely, a smooth quasiprojective surface $S$ with a smooth compactification $\overline{S}$.
We write classes in $H^{\mathrm{ev}}$ as $\alpha=\alpha^{(0)}+\alpha^{(1)}+\alpha^{(2)}$, where $\alpha^{(i)}\in H^{2i}(\overline{S})$.
Note that only $\alpha^{(2)}$ changes when we replace $\alpha$ by $\alpha+\delta$.
From now on, we will fix some $\alpha_0=\alpha_0^{(1)} + \alpha_0^{(2)}$ and consider only classes $\alpha \in \alpha_0 + \BZ \delta$.
In particular, for any moduli space $\mathcal{M}_\alpha$ we have $\rank \CF_\alpha=0$, and $\alpha^{(1)}$ is a class of an effective divisor.
For any $\eta \in H^2$ we have $\psi_0(\eta) = c_1(\CF_\alpha)\cdot\eta = \alpha^{(1)} \cdot \eta \in \mathbb{Q}$, and so there exists an effective divisor class $\eta\in H^2$ such that $\psi_0(\eta)>0$.
From now on we fix such $\eta$ and set
    \begin{equation}\label{eq:r-and-y}
    r \coloneqq \psi_0(\eta) = \alpha^{(1)} \cdot \eta > 0; \qquad y\coloneqq \frac{\psi_1(\eta)}{r}\in H^2(\mathcal{M}_\alpha).
    \end{equation}
We make the following further assumptions:
\begin{enumerate}[label=(A\arabic*),ref=A\arabic*]
    \setcounter{enumi}{2}
    \item \label{assumption-c} For any $\alpha$ there is an isomorphism $\phi:M_\alpha \times B\BG_m \simeq \CM_\alpha$ for some smooth scheme $M_\alpha$.
    \item \label{assumption-d} Tensoring by $\CO(\eta)$, where $\eta\in H^2$ is the element fixed above, induces isomorphisms $\CM_\alpha \simeq \CM_{\alpha+r\delta}$.
\end{enumerate}
The assumption~\eqref{assumption-c} implies the following isomorphism:
\begin{equation}\label{eq:trivial gerbe in cohomology}
    \phi^*:H^*(\CM_\alpha) \to H^*(M_\alpha\times B\BG_m)=H^*(M_\alpha)[v].
\end{equation}
Let us make one last technical assumption:
\begin{enumerate}[label=(A\arabic*),ref=A\arabic*]
    \setcounter{enumi}{4}
    \item \label{assumption-e} The element $v$ is a tautological class, that is $v\in H^2_\taut(\mathcal{M}_\alpha)$.
\end{enumerate}

The goal of the rest of this section is to extract a certain $\liesl_2$-action on $H^*_\taut(M_{\alpha})$ from the action of $W_S$ on $\V^{\taut}_S$, see~\eqref{eq:Vtaut-def}.
We achieve this by successive reductions:
\begin{itemize}
    \item First we show that $\V^{\taut}_S \simeq H^*_\taut(M_{\alpha})[y,s]$ (Sections~\ref{ssec:coarse}-\ref{ssec:period});
    \item Next, we get rid of $s$ and furnish $H_\taut^*(M_\alpha)[y]$ with an action of an intermediate algebra $\widetilde{W}_{\mathcal{M}}$ (Sections~\ref{sec:pol psi}-\ref{sec:def wt D});
    \item Reducing with respect to a certain Weyl subalgebra of $\widetilde{W}_{\mathcal{M}}$, we also get rid of $y$ and obtain an action of a smaller algebra $\Wred$ on $H^*_\taut(M_\alpha)$ (Section~\ref{ssec:reduction weyl});
    \item Finally, the desired $\liesl_2$-triple is found inside $\Wred$ (Section~\ref{ssec:sl2}).
\end{itemize}

\subsection{Coarse tautological classes}\label{ssec:coarse}
It follows by~\cite[Lemma 3.10]{heinloth2010lectures} from assumption~\eqref{assumption-c} that there exists a sheaf $F_\alpha$ on $M_\alpha\times S$ such that $(\phi\times\Id_S)^* \CF_\alpha\cong F_\alpha(1)$, where $(1)$ denotes tensoring with the weight one character of $\BG_m$.
Note that $F_\alpha$ depends on the choice of isomorphism $\phi$ and is therefore not canonical.
Using the sheaf $F_\alpha$, we can define the tautological classes $f(\gamma)\in H^*(M_\alpha)$ for any symmetric function $f$ and $\gamma\in H$ as in Section~\ref{subs:taut-classes}.
Denote the subring generated by these classes by $H^*_\taut(M_\alpha)\subset H^*(M_\alpha)$.
Using the equality $\ch(\mathcal{F}_\alpha) = \ch(F_\alpha)e^v$, we derive the following expressions for tautological classes of low degree:
\begin{align}
    \label{eq:y-p-taut} y = \frac{\psi_1(\eta)}{r} &= \frac{p_2(\eta)}{2r}+ \frac{1}{r}\int_{\overline{S}} c_1(F_\alpha) \cup (t_1+t_2) \cup \eta,\\
    \label{eq:v-p-taut} \frac{p_2(\eta)}{2r} &= \frac{1}{2r}\int_{\overline{S}} p_2(F_\alpha) \cup \eta + v.
\end{align}
In particular, the difference $y-v$ belongs to $H^2_\taut(M_\alpha)$.

\begin{lem}\label{lem:taut-coarse}
    Under the assumptions~\eqref{assumption-c} and~\eqref{assumption-e}, we have $H^*_\taut(\mathcal{M}_\alpha)\simeq H^*_\taut(M_\alpha)[y]$.
\end{lem}
\begin{proof}
    From the definition of $F_\alpha$, we have $\ch(\mathcal{F}_\alpha) = \ch(F_\alpha)e^v$; in particular, we have an inclusion $H^*_\taut(\mathcal{M}_\alpha)\subset H^*_\taut(M_\alpha)[v]$.
    By~\eqref{assumption-e} we have $v\in H^*_\taut(\mathcal{M}_\alpha)$, and so using the equality $\ch(F_\alpha) = \ch(\mathcal{F}_\alpha)e^{-v}$ we have the opposite inclusion $H^*_\taut(M_\alpha)[v]\subset H^*_\taut(\mathcal{M}_\alpha)$.
    Finally, as $H^*_\taut(M_\alpha)$ is a finite-dimensional graded algebra, the element $y-v\in H^2_\taut(M_\alpha)$ is nilpotent, so that $H^*_\taut(M_\alpha)[v] = H^*_\taut(M_\alpha)[y]$.
\end{proof}

\subsection{Periodicity}\label{ssec:period}
Observe that by assumption~\eqref{assumption-d}, tensoring with $\CO(\eta)$ induces an isomorphism $H^*(\CM_\alpha)\cong H^*(\CM_{\alpha+ r\delta})$ preserving the tautological rings.
The following result provides a canonical identification between the cohomology spaces $H^*_{\taut}(\CM_\alpha)$ for all $\alpha \in \alpha_0+\BN \delta$.
     \begin{prop}\label{prop:Hecke is iso}
         The Hecke operator $T_0(\eta)$ induces an isomorphism of vector spaces $H^*_\taut(\CM_\alpha)\cong H^*_\taut(\CM_{\alpha+\delta})$.
     \end{prop}
     \begin{proof}
    Recall that $T_0(\eta)$ sends $H^*_\taut$ to $H^*_\taut$, by the remark after~\eqref{eq:T from Q R}.
    Let us show that $T_0(\eta)$ is surjective; this will imply the chain of inequalities
    \[
         \dim H^i_\taut(\CM_\alpha) \geq \dim H^i_\taut(\CM_{\alpha+\delta}) \geq \cdots \geq \dim H^i_\taut(\CM_{\alpha+r\delta}),
         \]
         and we will deduce that $T_0(\eta)$ is an isomorphism from \eqref{assumption-d}.

         Let $f$ be a polynomial in the generators $\psi_i(\xi)$. By Theorem \ref{thm:deformed W} we know that some power $n$ of $\Ad_{T_0(\eta)}$ annihilates $f$. We show by induction on $n$ that $f(1)$ is in the image of $T_0(\eta)$. The induction base is not needed because the case $f=0$ is trivial. Now suppose that we have proven the claim for $[T_0(\eta), f]$, so that we have
         \[
         [T_0(\eta), f](1) = T_0(\eta)(g)
         \]
         for some $g\in H^*_\taut(\CM_{\alpha})$. Then we have
         \[
         T_0(\eta)(f(1) - g) = f  T_0(\eta)(1) = (\eta\cdot\alpha^{(1)}) f (1)=r f(1),
         \]
         hence $f(1)$ is in the image of $T_0(\eta)$.
     \end{proof}

Combining Proposition~\ref{prop:Hecke is iso}, Lemma~\ref{lem:taut-coarse} and \eqref{eq:trivial gerbe in cohomology} we obtain an isomorphism $\V^{\taut}_S \simeq H^*_\taut(M_{\alpha})[y,s]$, valid for any $\alpha \in \alpha_0 + \BN \delta$.

\subsection{Polynomiality of $\psi_n(\xi)$}\label{sec:pol psi}
Let us consider the family of tautological classes $\psi_n(\xi)$ for a fixed $\xi$, as $n$ varies.
Recall (see Remark~\eqref{eq:psi definition}) that 
\[
     \sum_{n\geq 0} \psi_n(\xi) = \int_{\overline{S}} \ch(\mathcal{F}_\alpha) \cup\Td_S\cup \xi.
\]
Rewriting the integrated expression, we get:
\begin{align*}
    \ch(\mathcal{F}_\alpha) \Td_S = \ch(F_\alpha) \Td_S e^v = \left(\ch(F_\alpha) \Td_S e^{-\frac{p_2(\eta)}{2r}+v}\right)e^{\frac{p_2(\eta)}{2r}}.
\end{align*}
We have $p_2(\eta)/2r - v \in H^2(M_\alpha)$ by~\eqref{eq:v-p-taut}.
In particular, the expression in the parentheses above belongs to $H^*(M_\alpha\times S)$.
Denoting it by $\sum_n \widetilde{Q}_n/n!$, $\widetilde{Q}_n\in H^{2n+2}(M_\alpha\times S)$, we obtain
\begin{align}\label{eq:psi-pre-prop}
    \psi_n(\xi) = \sum_{i=0}^n \binom{n}{i} \widetilde{Q}_i(\xi) \left(\frac{p_2(\eta)}{2r}\right)^{n-i}, \qquad \widetilde{Q}_i(\xi)\coloneqq \int_{\overline{S}}(\widetilde{Q}_i\cup \xi).
\end{align}
As $H^*(M_\alpha)$ is a finite-dimensional graded algebra, the element $\widetilde{Q}_n(\xi)$ is nilpotent when its cohomological degree is positive, and vanishes for $n$ large enough.

Recall the element $y\in H^2(\mathcal{M}_\alpha)$ from~\eqref{eq:r-and-y}.
Using~\eqref{eq:y-p-taut}, we see that $p_2(\eta)/2r - y \in H^2(M_\alpha)$.
Expanding the sum~\eqref{eq:psi-pre-prop} in the powers of $y$ and denoting the coefficients in this expansion by $Q_i(\xi)$, we arrive at the following.

\begin{prop}\label{prop:polynomiality1}
    There exists a family of elements $Q_i\in H^{2i+2}_\taut(M_\alpha\times S)$, $i\geq 0$ such that for any $n \geq 0$, we have the following equality in $H^*_\taut(\mathcal{M}_\alpha)$:
        \[
            \psi_n(\xi) = \sum_{i=0}^n \binom{n}{i} Q_i(\xi) y^{n-i},\qquad Q_i(\xi)\coloneqq (\pi_{M_\alpha})_*(Q_i\cup \xi).
        \]
    Moreover, $Q_n(\xi)$ is zero for $n$ large enough.\qed
\end{prop}

\subsection{Polynomiality of $D_{m,n}(\xi)$}\label{sec:pol D}
Let us now set
    \[
    X = \frac{T_0(\eta)}r = \frac{q_1(\eta)}{r}. 
    \]
By Proposition~\ref{prop:Hecke is iso}, the operator $X$ identifies all the spaces $H^*_\taut(\CM_{\alpha})$ for $\alpha \in \alpha_0+\BZ\delta$. Let us furthermore define
    \[
    D_{m,n}(\xi) \coloneqq \frac{n!}{(m+n)!}(-\Ad_{q_1(1)})^m \psi_{m+n}(\xi);
    \]
compare this with~\eqref{eq:def D m n undeformed}. These elements satisfy the properties listed in Theorem~\ref{thm:deformed W}.

 \begin{prop}\label{prop:D m n polynomial}
For any $\xi\in H$, $n\geq 0$ and any $\alpha \in \alpha_0+\BZ\delta$, the sequence of operators $X^{-m} D_{m,n}(\xi) \in \End(H^*_\taut(\CM_\alpha))$ depends polynomially on $m$. In the case $n=0$ the non-constant coefficients of the polynomial are nilpotent.
    \end{prop}
    \begin{proof}
    Using Proposition~\ref{prop:polynomiality1} we may write
    \[
    \psi_{m+n}(\xi) = \sum_{i=0}^N \binom{m+n}{i} y^{m+n-i} Q_i(\xi) ,
    \]
    where $N=\dim(M_\alpha)+2$. Let us denote $A \coloneqq -\Ad_{q_1(1)}$. Applying the Leibniz rule, we obtain
    \begin{equation}\label{eq:intermediate1}
    D_{m,n}(\xi) = \frac{n!}{(m+n)!} \sum_{i,j} \binom{m+n}{i} \binom{m}{j} \left(A^{m-j} y^{m+n-i}\right) \left( A^j Q_i(\xi)\right).
    \end{equation}
    Observe that $Q_i(\xi)$, viewed as a tautological class and hence as an element of $W_S$, belongs to $F_i$. By Theorem~\ref{thm:deformed W}, the summands of \eqref{eq:intermediate1} vanish for $j>i$. The sum is therefore finite, and we only need to check that for each fixed triple $(i,j,n)$ the sequence
    \[
     X^{j-m} \frac{n!}{(m+n)!} \binom{m+n}{i} \binom{m}{j} A^{m-j} y^{m+n-i}
    \]
    depends polynomially on $m$.
     Since $X$ and $q_1(1)$ commute, and $\binom{m}{j}$ is manifestly polynomial in $m$, it is enough to show the polynomiality of
    \begin{equation}\label{eq:to show}
     \frac{n!}{(m+n)!} \binom{m+n}{i} \left(X^{-1} A\right)^{m-j} y^{m+n-i}.
    \end{equation}
    Using Corollary~\ref{cor:q L}, we have $[y, q_1(1)] = r^{-1} [\psi_1(\eta), q_1(1)] = r^{-1} q_1(\eta) = X$ and similarly $[y, X] = r^{-2} q_1(\eta^2)$, which commutes with $X$ and $y$.
    We deduce that for any $k\geq 0$
    \[
    -X^{-1} \Ad_{q_1(1)} y^k = X^{-1} \sum_{i=0}^{k-1} y^i X y^{k-1-i} = k y^{k-1} + \binom{k}2 X^{-1}[y,X] y^{k-2}
    \]
    \[
    = \left(\frac{\partial}{\partial y} + \frac{1}{2}X^{-1}[y,X] \frac{\partial^2}{\partial^2 y}\right) y^k.
    \]
    Write $z=\frac{1}{2}X^{-1}[y,X]$. The binomial formula gives
    \[
    (-X^{-1} \Ad_{q_1(1)})^{m-j} y^{m+n-i} = \sum_{k=0}^{\min(m-j, n-i+j)} \binom{m-j}k z^k \frac{(m+n-i)!}{(n-i+j-k)!} y^{n-i+j-k}.
    \]
    The upper bound of the summation can be replaced by $n-i+j$. As a result we obtain that \eqref{eq:to show} is a linear combination of operators of the form $z^k y^{n-i+j-k}$, and each such operator enters with a coefficient equal to
    \[
    \frac{n!}{(m+n)!} \binom{m+n}{i} \binom{m-j}k \frac{(m+n-i)!}{(n-i+j-k)!} = \frac{n!}{i!(n-i+j-k)!} \binom{m-j}k,
    \]
    which is a polynomial in $m$.

    Let us now consider the special case $n=0$. The inequalities $j\leq i$ and $0\leq k \leq n-i+j$ imply $k=0$ and $i=j$. So the result is a linear combination of terms of the form
    \[
    \frac{1}{i!} \binom{m}{i} \left( (-\Ad_{q_1(1)})^i Q_i(\xi)\right).
    \]
    Only the terms with $i\geq 1$ contribute to the non-constant coefficients of the polynomial. If $Q_i(\xi)$ has positive cohomological degree, then it is nilpotent, say $Q_i(\xi)^M=0$. Applying $(-\Ad_{q_1(1)})^{Mi}$ and using $(\Ad_{q_1(1)})^{i+1} Q_i(\xi)=0$ we obtain $\left((-\Ad_{q_1(1)})^i Q_i(\xi)\right)^M=0$. If $Q_i(\xi)$ has cohomological degree zero (which only happens for $i=1$, $\xi\in H^0$), then $(-\Ad_{q_1(1)})^i Q_i(\xi)$ has negative cohomological degree as an operator and is therefore nilpotent.
    \end{proof}

    \subsection{Construction of $\wt D_{m,n}(\xi)$ and relations}\label{sec:def wt D}
    Proposition~\ref{prop:D m n polynomial} gives rise to operators on each individual $H^*_\taut(\CM_\alpha)$, by taking the coefficients in the polynomial expansion of $X^{-m}D_{m,n}(\xi)$ in $m$. We will need a slight renormalization.
    Let $\theta$ be the linear term in the polynomial expansion of $X^{-k} q_k(\eta)$. By Proposition~\ref{prop:D m n polynomial} $\theta$ is nilpotent, and thus $e^{k\theta}$, $e^{-k\theta}$ are polynomials in $k$. So Proposition~\ref{prop:D m n polynomial} remains valid when $X^{-k}$ is replaced by $u^{-k}$, where
    \[
    u= X e^{\theta/q_0(\eta)}=X e^{\theta/r}.
    \]
    This choice of $u$ ensures that the linear term in the expansion of $u^{-k} q_k(\eta)$ vanishes. Let us now define operators $\wt D_{i,n}(\xi)$ for all $n,m \in \BN$ and $\xi \in H$ by
    \begin{equation}\label{eq:def td D n m}
    D_{m,n}(\xi) = u^m \sum_i \frac{m^i}{i!} \wt D_{i,n}(\xi).
    \end{equation}
    In particular, note that $\psi_n(\xi)=D_{0,n}(\xi)=\wt D_{0,n}(\xi)$.
    As in Section~\ref{sec:undeformed}, we set 
    $$\wt q_i(\xi)=\wt D_{i,0}(\xi), \qquad \wt L_i(\xi)=\wt D_{i,1}(\xi)$$
    for $i \geq 0$; observe that by construction $\wt q_1(\eta)=0$ and $\wt q_i(\xi) = \wt L_i(\xi) = 0$ for $i\gg 0$. Define $\wt W_{\mathcal{M}}$ as the subalgebra of $\End(H^*_\taut(\CM_\alpha))$ generated by all the operators $\wt D_{m,n}(\xi)$. Finally, denote by $\wt F_n$ the filtration of  $\wt W$ obtained by placing $\wt D_{m,n}(\xi)$ in weight $n$.


\begin{prop}\label{prop:relations rational}
The following relation holds modulo $\wt F_{n+n'-3}$:
\begin{equation*}
\begin{split}
    [\wt D_{m,n}(\xi)&, \wt D_{m',n'}(\xi')] =(n m'-m n') \wt D_{m+m'-1,n+n'-1}(\xi\xi')\\
    &- r^{-1} n m' \wt D_{m,n-1}(\xi\eta) \wt D_{m'-1,n'}(\xi') 
\pm r^{-1} m n' \wt D_{m',n'-1}(\xi'\eta) \wt D_{m-1,n}(\xi)\\
&- 2 r^{-2} \binom{n}2\binom{m'}2 \wt D_{m,n-2}(\xi\eta^2) \wt D_{m'-2,n'}(\xi')\pm 2  r^{-2} \binom{m}2 \binom{n'}2 \wt D_{m',n'-2}(\xi'\eta^2) \wt D_{m-2,n}(\xi),
\end{split}
\end{equation*}
where the sign is $-$ if both $\xi$ and $\xi'$ are odd. When $s_2 = c_1\Delta_S = 0$, the relations hold on the nose.
\end{prop}
\begin{proof}
Let us first find commutation relations involving $u$. From $[\psi_1(\xi), q_k(\eta)] = k q_{k}(\xi\eta)$ and \eqref{eq:def td D n m} we obtain
    \[
    k u^{k-1} [\psi_1(\xi), u] \sum_i \wt q_i(\eta) \frac{k^i}{i!} + u^k \sum_i [\psi_1(\xi), \wt q_i(\eta)] \frac{k^i}{i!} = k u^{k} \sum_i \wt q_i(\xi \eta) \frac{k^i}{i!}.
    \]
Dividing by $u^k$ and taking the linear terms in $k$ we obtain
    \[
        q_0(\eta) u^{-1} [\psi_1(\xi), u] + [\psi_1(\xi), \wt q_1(\eta)] = q_0(\xi\eta).
    \]
    As $\wt q_1(\eta)=0$, this yields
    \begin{equation}\label{eq:psi-u-comm}
        [\psi_1(\xi), u] = u \frac{q_0(\xi\eta)}{r} = u \frac{\psi_0(\xi\eta)}{r}.
    \end{equation}
    Recall the operator $\frakd = \frac{D_{0,2}(1)}{2}$.
    Since $\Ad_{q_i(\xi)}$ lowers the $F$-filtration degree by $1$, we have $\Ad_{q_i(\xi)}^3 \frakd=0$, and so in particular
    \[
        [\frakd, u^k] = k u^{k-1} [\frakd, u] + \binom{k}{2} u^{k-2} [[\frakd,u],u].
    \]
    Expanding $q$ in terms of $\wt q$ in $[\frakd, q_k(\eta)] = k L_{k}(\eta)$ and applying the equation above, we obtain
    \[
        \left(k u^{k-1} [\frakd, u] + \binom{k}{2} u^{k-2} [[\frakd,u],u]\right) \sum_i \wt q_i(\eta) \frac{k^i}{i!} + u^k \sum_i [\frakd, \wt q_i(\eta)]\frac{k^i}{i!}
        = k u^{k} \sum_i \wt L_i(\eta) \frac{k^i}{i!}.
    \]
Dividing by $u^k$, using the relations $\wt q_0(\eta)=r, \wt L_0(\eta)=L_0(\eta)$ and collecting the coefficient of $k^1$ produces
\[
u^{-1} [\frakd, u] - \frac12  u^{-2} [[\frakd,u],u] = \frac{L_0(\eta)}r = \frac{\psi_1(\eta)}r.
\]
Applying $-\Ad_u$ and using~\eqref{eq:psi-u-comm} we thus obtain
\[
u^{-2} [[\frakd, u], u] = \frac{\psi_0(\eta^2)}{r^2} \qquad\Rightarrow\qquad u^{-1} [\frakd, u] = \frac{\psi_1(\eta)}{r} + \frac12\frac{\psi_0(\eta^2)}{r^2}.
\]
Using $[\frakd, D_{m,n}(\xi)] = m D_{m,n+1}(\xi)\mod F_{n-1}$ we obtain by induction on $n$
\[
[D_{m,n}(\xi), u] = u \left( n \frac{D_{m,n-1}(\xi\eta)}r + \binom{n}2 \frac{D_{m,n-2}(\xi\eta^2)}{r^2}\right)\mod F_{n-3},
\]
and then by induction on $k$
\[
[D_{m,n}(\xi), u^k] = u^k \left( k n \frac{D_{m,n-1}(\xi\eta)}r + k^2 \binom{n}2 \frac{D_{m,n-2}(\xi\eta^2)}{r^2}\right)\mod F_{n-3}.
\]
Since this holds for all $m\geq 1$, by linearity the same identity is true with $\wt D$ in place of $D$.

Finally we expand the identity~\eqref{eq:W relations} modulo $F_{n+n'-3}$:
\[
(m' n - m n') u^{m+m'} \sum_k \wt D_{k,n+n'-1}(\xi)\frac{(m+m')^k}{k!}
= \sum_{i,j}\frac{m^i}{i!}\frac{m'^j}{j!} \left[u^m \wt D_{i,n}(\xi), u^{m'} \wt D_{j,n'}(\xi')\right].
\]
The commutator expands as follows:
\begin{align*}
    u^{m+m'}\Big([\wt D_{i,n}(\xi),& \wt D_{j,n'}(\xi')]
+ \left(\frac{m' n}r \wt D_{i,n-1}(\xi\eta) + \frac{m'^2}{r^2} \binom{n}2 \wt D_{i,n-2}(\xi\eta^2)\right) \wt D_{j,n'}(\xi')\\
&\mp \left(\frac{m n'}r \wt D_{j,n'-1}(\xi'\eta) + \frac{m^2}{r^2} \binom{n'}2 \wt D_{j,n'-2}(\xi'\eta^2)\right)  \wt D_{i,n}(\xi)\Big),
\end{align*}
where the sign is $+$ if both $\xi$ and $\xi'$ are odd. Collecting monomials in $m$, $m'$ yields the first claim.
Finally, if $s_2 = c_1\Delta_S = 0$, the identity~\eqref{eq:W relations} holds on the nose by Theorem~\ref{thm:W undeformed}, which proves the second claim.
\end{proof}

\begin{cor}The algebra $\wt W_{\mathcal{M}}$ is generated by $\{\psi_2(1), \wt q_k(\xi)\,|\, k \geq 0, \xi \in H\}$.\qed
\end{cor}

\subsection{Reduction with respect to a Weyl subalgebra}\label{ssec:reduction weyl}
We now come to the second reduction step. Observe that we have
\[
[\psi_1(\eta), \wt q_1(1)] = [\wt D_{0,1}(\eta), \wt D_{1,0}(1)] = D_{0,0}(\eta) - r^{-1} D_{0,0}(\eta^2) D_{0,0}(1) = r
\]
(recall that $D_{0,0}(1)=\psi_0(1)=0$ for degree reasons). Thus the pair of elements 
$$y=\frac{\psi_1(\eta)}{r}, \qquad \partial_y\coloneqq-\wt q_1(1)$$ generate a Weyl subalgebra of $\wt W_{\mathcal{M}}$, i.e. we have $[\partial_y, y]=1$. 

\begin{lem}\label{lem: weyl acts nilp}
The operator $\partial_y$ acts locally nilpotently on $H^*_\taut(\CM_\alpha)$.
The operators $\Ad_y$ and $\Ad_{\partial_y}$ act locally nilpotently on $\wt W_{\mathcal{M}}$.
\end{lem}
\begin{proof}
The first statement is a consequence of the fact that $\partial_y$ is of cohomological degree $-2$. Next, because $\partial_y\in\wt F_0$, the operator $\Ad_{\partial_y}$ acts locally nilpotently on $\wt W_S$. Finally, from the relations in Proposition~\ref{prop:relations rational} one sees that $\Ad_y$ acts nilpotently on $\wt F_m / \wt F_{m-2}$, and hence locally nilpotently on $\wt W_S$.
\end{proof}

Consider the following general setup. Let $A$ be an associative algebra defined over a field $\mathbb{K}$ of characteristic $0$, acting on a $\mathbb{K}$-vector space $V$. Assume that the Weyl algebra $B=\mathbb{K} y \oplus \mathbb{K} \oplus \mathbb{K} \partial_y$ is a Lie subalgebra of $A$, that $\Ad_y$ and $\Ad_{\partial_y}$ both act locally nilpotently on $A$ and $\partial_y$ acts locally nilpotently on $V$.

\begin{lem}\label{lem:reduction weyl general} Let $A,B,V$ be as above. Then
\begin{enumerate}
    \item Multiplication induces an isomorphism of vector spaces $m:\mathbb{K}[y] \otimes A_{\red} \otimes \mathbb{K}[\partial_y] \xrightarrow{\sim} A$, where $A_{\red}$ is the subalgebra of $A$ of elements commuting with both $y$ and $\partial_y$,
    \item The action map defines an isomorphism of vector spaces $V_\red [y] \xrightarrow{\sim} V$, where $V_\red=\kernel(\partial_y)$,
    \item The action of $A$ on $V$ restricts to an action of $A_\red$ on $V_\red$.
    \item The natural projections 
    $A \to A_\red, \; f=\sum_{i,j}f_{i,j}y^i\partial_y^j \mapsto f_\red\coloneqq f_{0,0}$ and
    $V \to V_\red, \; v=\sum_i y^i v_i \mapsto v_\red \coloneqq v_0 $
    are given by the following formulas
    \begin{equation}\label{eq:formula expansion}
    f_\red=\sum_{i,j}\frac{1}{i! j!} y^i\left( \Ad_{y}^j (-\Ad_{\partial_y})^i f\right)  \partial_y^j , \qquad v_\red=\sum_i \frac{y^i(-\partial_y)^i}{i!}v\ .
    \end{equation}
\end{enumerate}
\end{lem}
\begin{proof}We will first show that the map $m:\mathbb{K}[y] \otimes A_\red\otimes \mathbb{K}[\partial_y] \to A$ is both injective and surjective. To begin, a standard argument via applying $\Ad_{y}$ shows that the multiplication map $m_1:A_\red\otimes \mathbb{K}[\partial_y] \to\kernel(\Ad_{\partial_y})$ is injective. To see that it is also surjective we argue by induction on $n$ to prove that $\kernel(\Ad_{\partial_y}) \cap \kernel(\Ad_y^n) \subset \image(m_1)$, using the observation that for $f \in \kernel(\Ad_y^n) \backslash \kernel(\Ad_y^{n-1})$ and $u=\frac{1}{(n-1)!}\Ad_y^{n-1}f$ we have $\Ad_y^{n-1}(u \partial_y^{n-1}-f)=0$. We conclude using the fact that $\Ad_y$ is locally nilpotent. The same argument with $y$ and $\partial_y$ swapped implies that $m$ is an isomorphism. Statement (ii) is classical, and shown in the same fashion; statement (iii) immediately follows. For $v = y^a u$ with $u \in V_\red$, we compute
$$\sum_i \frac{y^i (-\partial_y)^i}{i!}v=\sum_{i=0}^a (-1)^i \frac{a!}{(a-i)!i!}y^a u=\sum_{i=0}^a (-1)^i \binom{a}{i} y^a u=\delta_{a,0}u,$$
proving the second formula of (iv). The first equality of (iv) is shown by a double application of the second one.
\end{proof}

In our situation, Lemma~\ref{lem:reduction weyl general} yield canonical decompositions
\begin{equation}\label{eq:identi stack reduced coarse}
H^*_\taut(\CM_\alpha) = H^*_\taut(\CM_\alpha)_\red [y], \qquad \wt W_{\mathcal{M}} = \Wred[y][\partial_y]
\end{equation}
where $H^*_\taut(\CM_\alpha)_\red = \kernel \partial_y$, and $\Wred$ consists of operators commuting with $y$ and $\partial_y$.
Moreover, $\Wred$ preserves $H^*_\taut(\CM_\alpha)_\red$. 
By Lemma~\ref{lem:taut-coarse}, we have identifications
$$H_\taut^*(\CM_\alpha)_\red \simeq H_\taut^*(\CM_\alpha)/ yH_\taut^*(\CM_\alpha) \simeq H_\taut^*(M_\alpha).$$ 
We have thus obtained an action of $\Wred$ on $H_\taut^*(M_\alpha)$.

\begin{rem}\label{rmk:killing-y}
There are two ways of viewing $H_\taut^*(\CM_\alpha)_\red\simeq H_\taut^*(M_\alpha)$: as the subspace $\kernel \partial_y\subset H_\taut^*(\CM_\alpha)$ or as the quotient space $H_\taut^*(\CM_\alpha) \twoheadrightarrow \cokernel y$. Since $\psi_m(\xi)$ commutes with $y$, we have
\[
\psi_m(\xi)_\red = \sum_{i} \frac{y^i(-\Ad_{\partial_y})^i}{i!} \psi_m(\xi),
\]
and the action of $\psi_m(\xi)_\red$ is compatible with that of $\psi_m(\xi)$ via the quotient identification. 
Because of this, we will omit subscript ``$\red$'' for the tautological classes in $H^*(M_\alpha)$.
On the other hand, $\wt q_m(\xi)$ commutes with $\partial_y$, and so we have
\[
\wt q_m(\xi)_\red = \sum_{i} \left(\Ad_{y}^i \wt q_m(\xi)\right) \frac{\partial_y^i }{i!}.
\]
Since $\partial_y$ annihilates $H_\taut^*(\CM_\alpha)_\red$, the action coincides with the induced action of $\wt q_m(\xi)$ with respect to the subspace identification. In general, for any operator $f$ the induced action from $\kernel \partial_y$ to $\cokernel y$ coincides with $f_\red$.
\end{rem}

\subsection{$\liesl_2$-triple}\label{ssec:sl2}
Let us extract an $\liesl_2$-triple in $\End(H_\taut^*(M_\alpha))$, which will play a key role later, out of the algebra $\Wred$. Set $\frakh$ to be the operator $-\wt D_{1,1}(1)_\red$. Explicitly, by~\eqref{eq:formula expansion} we have
\[
\frakh = -\wt D_{1,1}(1) - (y - r^{-2}\psi_0(\eta^2) \psi_1(1)) \partial_y.
\]
Taking into account that $\wt D_{1,0}(\eta) = \wt q_1(\eta)=0$, we obtain by Proposition~\ref{prop:relations rational} that 
\[
[\wt D_{1,1}(1), \wt D_{m,n}(\xi)] = (m-n) \wt D_{m,n}(\xi) + r^{-1} n \wt D_{m,n-1}(\eta\xi) \psi_1(1)\mod \wt F_{n-2}.
\]
In particular, using commutation relations between $\psi_1(1)$ and $\wt q_k(\xi)$ we get
\[
[\frakh, -\wt q_m(\xi)] = -m \wt q_m(\xi) - [y, \wt q_m(\xi)] \partial_y,
\]
and taking the constant terms
\begin{equation}\label{eq:h-psired-comm}
    [\frakh, \wt q_m(\xi)_\red] = -m \wt q_m(\xi)_\red.    
\end{equation}
As a consequence, the following holds.
\begin{prop}
    The operator $\Ad_{\frakh}$ restricted to $\wt F_0\cap W_{\red}$ is diagonalizable and its eigenvalues are non-positive integers.\qed
\end{prop}

Consider
\[
\frakd_\red = \frac{\wt D_{0,2}(1)_\red}2 = \frac{\psi_2(1)}{2} - y \psi_1(1).
\]
As $\psi_1(1)\in H^0(\CM_\alpha)$ is a scalar, it commutes with all operators. Since $\frakd_\red$ commutes with both $y$ and $\partial_y$ we have
\[
[\frakh, \frakd_\red] =  -\left[\wt D_{1,1}(1), \frac{\psi_2(1)}{2} - y \psi_1(1)\right] =  \psi_2(1) - 2 y \psi_1(1) - r^{-2} \psi_0(\eta^2)\psi_1(1)^2 = 2\frakd_\red \mod{\wt F_0}.
\]
Write
\[
[\frakh, \frakd_\red] = 2\frakd_\red + \sum_{i\geq 0} f_i,
\]
where each $f_i\in \wt F_0\cap \Wred$ is an eigenvector of $\Ad_{\frakh}$ of eigenvalue $(-i)$, and set
\begin{equation}\label{eq:frake}
    \frake = \frakd_\red + \sum_{i}\frac{f_i}{i+2}.
\end{equation}
Then we have
\[
[\frakh, \frake] = 2\frakd_\red + \sum_{i\geq 0} \left(1-\frac{i}{i+2}\right)f_i = 2 \frake.
\]
Finally, set $\frakf = -\frac{\wt q_2(1)_\red}{2}$. From Proposition~\ref{prop:relations rational}, we have
\[
[\frake, \wt q_2(1)] = [\frakd_\red, \wt q_2(1)] = 2 \wt D_{1,1}(1) - 2 r^{-1} \wt D_{0,1}(\eta) \wt D_{1,0}(1) + 2 r^{-2}\psi_0(\eta)^2 \wt D_{1,0}(1) \psi_1(1).
\]
Taking the constant terms we obtain $[\frake, \wt q_2(1)_\red] = -2 \frakh$, and hence we have
\[
[\frake, \frakf] = \frakh.
\]

We summarize the above discussion:
\begin{thm}\label{thm:sl2-general}
\begin{enumerate}
    \item The operators $(\frake,\frakh, \frakf)$ form a $\liesl_2$-triple in $\End(H^*_\taut(M_\alpha))$, for any $\alpha \in \alpha_0+\BZ \delta$.
    \item Define a filtration $P_\bullet$ of $H^*_\taut(M_\alpha)$ by setting $P_i$ to be the sum of the $\frakh$-eigenspaces with eigenvalues $\leq i$. Then:
    \begin{enumerate}
        \item the tautological classes satisfy $\psi_k(\xi) P_i\subset P_{i+k}$,
        \item the Hecke operators satisfy $\wt q_{k}(\xi)_\red P_i\subset P_{i-k}$,
        \item on the associated graded, the operator $\psi_2(1)$ coincides with $2\frake$ and therefore satisfies the Hard Lefschetz property, that is $\psi_2(1)^k:P_{-k}/P_{-k-1} \to P_{k}/P_{k-1}$ is an isomorphism.
    \end{enumerate}
        \item For any sequence of open subspaces $\CM_{\alpha_0+\BZ\delta}'\subset \CM_{\alpha_0+\BZ\delta}$ preserved by the Hecke correspondences the restriction maps commute with the $\liesl_2$-action and are therefore strictly compatible with the filtrations $P'_\bullet, P_\bullet$.
    \end{enumerate}
\end{thm}
\begin{proof}
    All statements save for (ii.a-b) follow from our construction.
    The statement (ii.b) follows from the commutation relation~\eqref{eq:h-psired-comm}.
    Finally, for (ii.a) we use Proposition~\ref{prop:relations rational}.
    Since the relation $[D_{1,1}(1),\psi_k(\xi)]$ holds on the nose in $W_S$, the corresponding relation in $\wt W_{\mathcal{M}}$ holds on the nose as well:
    \[
        [\wt D_{1,1}(1), \psi_k(\xi)] = -k\psi_k(\xi) + r^{-1}k\psi_{k-1}(\xi\eta)\psi_1(1).
    \]
    Let $f$ be an eigenvector of $\frakh$ with eigenvalue $i$. We have
    \begin{align*}
        \frakh(\psi_k(\xi) f) &= \psi_k(\xi)\frakh f + [\frakh,\psi_k(\xi)]f
        = i(\psi_k(\xi) f) - [\wt D_{1,1}(1),\psi_k(\xi)]f\\
        &= (i+k)\psi_k(\xi) f - (r^{-1}k\psi_1(1))\psi_{k-1}(\xi\eta)f.
    \end{align*}
    Reasoning by induction on $k$, the second term belongs to $P_{i+k-1}$, and so we conclude that $\psi_k(\xi) f\in P_{i+k}$.
\end{proof}
For any $\liesl_2$-triple acting on $H^*(M_\alpha)$ we will refer to the filtration $P_\bullet$ defined in (ii) as the \textit{$\frakh$-degree filtration}.

\section{Higgs bundles}\label{sec:Higgs-bun}
From now on and until the end of the paper, we restrict to the following situation. Let $C$ be a smooth projective curve of genus $g\geq 0$, and let $D$ be a (possibly zero) effective simple divisor. We denote by $S=\mathbb{P}(\Omega(D) \oplus \mathcal{O})$ the projective completion of the line bundle $\Omega(D)$ over $C$, and let $\pi_C:S\to C$ be the projection. We also let $S^\circ \subset S$ be the total space of $\Omega(D)$. Recall that a \textit{$D$-twisted Higgs bundle} on $C$ consists of a pair $(\CE, \theta)$ where $\CE$ is a vector bundle on $C$ and $\theta:\CE\to\CE \otimes \Omega(D)$. By the Beauville--Narasimhan--Ramanan correspondence~\cite{BNR} one may equivalently view $(\CE,\theta)$ as a purely one-dimensional coherent sheaf $\CF$ on $S$ whose support is disjoint from the boundary divisor $\partial S=S \backslash S^\circ$; the equivalence is given by $\CE=\pi_*(\CF)$. The support of $\CF$ is called the \textit{spectral curve} of $(\CE,\theta)$.

Let us fix a basis 
\[
\Pi=\{1, \gamma_1,\ldots,\gamma_{2g}, \omega\},
\]
of the cohomology ring $H=H^*(C)$ in which the product is given by
\[
1\cdot x = x \cdot 1 = x\quad(x\in\Pi),\qquad \gamma_i \gamma_{i+g} = -\gamma_{i+g} \gamma_i = \omega\quad(1\leq i\leq g),
\]
and all other products of generators vanish.
Note that we have canonical isomorphisms 
$H^*_\pure(S^\circ) \simeq H^*(S^\circ) \simeq H^*(C)=H$. 

\begin{rem}\label{rmk:ci-from-rd}
We have $H^*(S)=H^*(C)[\xi]/\langle \xi^2+(|D|+2g-2)\omega \xi\rangle$, where $\xi=c_1(\CO(\partial S))$, and the restriction map $H^*(S)\to H^*(S^\circ)\simeq H$ is given by setting $\xi \mapsto 0$.
If $\CF$ corresponds to a Higgs bundle $(\CE,\theta)$ of rank $r$ and degree $d$, then
$$c_1(\CF)=r(\xi+(|D|+2g-2)\omega), \qquad c_2(\CF)=(r(r+1)(1-g)-d)\omega\xi.$$
It will be more convenient to work with $(r,d)$, so we will not use these formulas.
\end{rem}

In this section, we specialize the results of Section~\ref{sec:sl2} to twisted Higgs bundles and their parabolic versions.

\subsection{Algebras in the twisted case}
We apply the construction of Section \ref{sec:hecke} to the following moduli space. Fix $r>0$. For each $d\in\BZ$, let $\CM_{r,d,D}^\ellipt$ be the moduli stack of $D$-twisted Higgs bundles of rank $r$ and degree $d$ whose spectral curve is reduced and irreducible. We call it the \textit{elliptic} moduli stack.

It is well-known that $\CM_{r,d,D}^\ellipt$ is of finite type and is a $\BG_m$-gerbe over its coarse moduli space $M_{r,d,D}^\ellipt$, which is smooth.
More precisely, it is the relative Jacobian over the family of spectral curves; see~\cite[Prop. 4.3, 5.2]{ngo2006fibration}. 
Note that Hecke modifications and tensoring by line bundles change the degree $d$ but not the support of purely one-dimensional sheaves, hence assumptions (\ref{assumption-a}-\ref{assumption-b}) of Section~\ref{ssec:hecke notations} as well as (\ref{assumption-c}-\ref{assumption-d}) of Section~\ref{sec:sl2} are satisfied, where we set $\eta = \omega$.
For the assumption~\eqref{assumption-e}, we need a short argument.

\begin{lem}
    We have $v\in H^2_{\taut}(\CM_{r,d,D}^\ellipt)$.
\end{lem}
\begin{proof}
    Thanks to (\ref{eq:y-p-taut}-\ref{eq:v-p-taut}), it suffices to express the class $\int_{\overline{S}} p_2(F_\alpha)\cup \omega$ in terms of tautological classes in $\mathcal{F}_\alpha$.
    As explained in Lemma~\ref{lem:psi-GRR}, we have $p_2(F_\alpha) = c_1(E_\alpha)$ and $p_2(\mathcal{F}_\alpha) = c_1(\mathcal{E}_\alpha)$, where $E_\alpha$, resp. $\mathcal{E}_\alpha$ is the universal vector bundle on $M_{r,d,D}^\ellipt\times C$, resp. $\CM_{r,d,D}^\ellipt\times C$.
    Consider the map $f:M_{r,d,D}^\ellipt\to \mathrm{Pic}^d(C)$, which sends $(\CE, \theta)$ to $\det \CE$.
    By functoriality of Chern classes, we have $c_1(E_\alpha) = (f\times \id_C)^*c_1(L)$, where $L$ is the tautological line bundle on $\mathrm{Pic}^d(C) \times C$.
    Finally, it is well known that $H^*(\mathrm{Pic}^d) = \Lambda^\bullet H^1(\mathrm{Pic}^d)$, and $H^1(\mathrm{Pic}^d)$ is spanned by the tautological classes, see e.g.~\cite[\S 4.1]{heinloth2010lectures}.
    Hence $\int_C c_1(L)\cup\omega$ can be expressed in terms of classes $\int_C c_1(L)\cup\gamma_i$.
    Pulling back, we see that $\int_{\overline{S}} p_2(F_\alpha)\cup \omega$ is expressed in terms of classes
    \[
        \int_C c_1(E_\alpha)\cup\gamma_i = \int_C c_1(\mathcal{E}_\alpha)\cup\gamma_i = \int_{\overline{S}} p_2(\mathcal{F}_\alpha)\cup\gamma_i = p_2(\gamma_i),
    \]
    where the first equality follows from $H^1(M_{r,d,D}^\ellipt) = H^1(\CM_{r,d,D}^\ellipt)$. 
    The desired statement follows.
\end{proof}

By Theorems~\ref{thm:surface W relations} and~\ref{thm:W undeformed} we have
\begin{cor}\label{cor:trigonometric}
    For fixed $r$ and $D$, there is an action of the (super) Lie algebra with basis $D_{m,n}(\pi)$, $m,n\geq 0$, $\pi\in\Pi$ and Lie bracket
    \[
    [D_{m,n}(\pi), D_{m',n'}(\pi')] = (m'n-m n') D_{m+m', n+n'-1}(\pi\pi')
    \]
    on the direct sum
    \[
    \bigoplus_{d\in\BZ} H^*_\taut(\CM_{r,d,D}^\ellipt).
    \]
    For $\pi\in H^i$, $D_{0,n}(\pi)$ is the operator of cup product by the tautological class $\psi_n(\pi)\in H^{2n-2+i}(\CM_{r,d,D}^\ellipt)$ and
    \[
    D_{1,n}(\pi) = T_n(\pi): H^{j}(\CM_{r,d-1,D}^\ellipt) \to H^{j+2n-2+i}(\CM_{r,d,D}^\ellipt)
    \]
    is the Hecke operator.
\end{cor}

Next, we make explicit the results of Section~\ref{sec:sl2}. Set $\eta = \omega$ so that $\alpha^{(1)}\cdot\eta = r$, i.e. the number $r$ in Section \ref{sec:sl2} is precisely the rank of Higgs bundles under consideration. The relations of the algebra $\wt W_{\mathcal{M}}$, which are given in Proposition \ref{prop:relations rational}, hold not only modulo $F_{n+n'-3}$, but on the nose by Remark~\ref{rem:undeformed}.

\begin{cor}\label{cor:rational reduced}
For fixed $r$, $d$ and $D$, there is an action of the algebra generated by $\wt D_{m,n}(\pi)$, $m,n\geq 0$, $\pi\in\Pi$ modulo relations
		\[
[\wt D_{m,n}(\pi), \wt D_{m',n'}(\pi')] = (n m'-m n') \wt D_{m+m'-1,n+n'-1}(\pi\pi')
\]
\[
- r^{-1} n m' \wt D_{m,n-1}(\pi\omega) \wt D_{m'-1,n'}(\pi')
+ r^{-1} m n' \wt D_{m',n'-1}(\pi'\omega) \wt D_{m-1,n}(\pi)
\]
on $H^*_\taut(\CM_{r,d,D}^\ellipt)$,
where for any $\pi\in H^i$, $\wt D_{0,n}(\pi) = D_{0,n}(\pi)$ is the operator of cup product $\psi_n(\pi)$.\qed
\end{cor}


We next move on to the reduced version $\Wred$ of $\wt W_{\mathcal{M}}$. Recall that the elements $y = \frac{\psi_1(\omega)}r$, $\partial_y = -\wt D_{1,0}(1)$ generate a Weyl algebra and induce a canonical decomposition
\[
H^*_\taut(\CM_{r,d,D}^\ellipt) = H^*_\taut(\CM_{r,d,D}^\ellipt)_\red[y],
\]
where $H^*_\taut(\CM_{r,d,D}^\ellipt)_\red$ is the subring annihilated by $\partial_y$. From the relation
\[
[\partial_y, \psi_n(\pi)] = [\wt D_{0,n}(\pi), \wt D_{1,0}(1)] = n \wt D_{0,n-1}(\pi) = n \psi_{n-1}(\pi)
\]
it follows that coefficients of the generating series
\[
\sum_{n=0}^\infty \frac{\psi_n(\pi)_\red}{n!} = e^{-y} \sum_{n=0}^\infty \frac{\psi_n(\pi)}{n!}
\]
generate $H^*_\taut(\CM_{r,d,D}^\ellipt)_\red$. Recall the identification
\[
H^*_\taut(\CM_{r,d,D}^\ellipt)_\red = H^*_\taut(M_{r,d,D}^\ellipt)
\]
(see \eqref{eq:identi stack reduced coarse}). The subspace $H^*_\taut(\CM_{r,d,D}^\ellipt)_\red$ is preserved by the operators $\wt D_{m,n}(\pi)_\red$ which commute with both $\partial_y$ and $y$; see Section~\ref{ssec:reduction weyl}.

\begin{prop}\label{prop:rational reduced twice}
    For fixed $r$, $d$ and $D$, the algebra generated by $\wt D_{m,n}(\pi)_\red$, $m,n\geq 0$, $\pi\in\Pi$ modulo relations
\begin{equation}\label{eq:ws red higgs}
\begin{split}
    [\wt D_{m,n}&(\pi)_\red, \wt D_{m',n'}(\pi')_\red] = (n m'-m n') \wt D_{m+m'-1,n+n'-1}(\pi\pi')_\red\\
   & - r^{-1} n m' \left(\wt D_{m,n-1}(\pi\omega)_\red \wt D_{m'-1,n'}(\pi')_\red + \wt D_{m'-1,n'}(\pi'\omega)_\red \wt D_{m,n-1}(\pi)_\red\right)\\
    & + r^{-1} m n' \left(\wt D_{m-1,n}(\pi\omega)_\red \wt D_{m',n'-1} (\pi')_\red + \wt D_{m',n'-1}(\pi'\omega)_\red \wt D_{m-1,n}(\pi)_\red\right).
\end{split}
\end{equation}
acts on $H^*_\taut(M_{r,d,D}^\ellipt)$,
where $\wt D_{0,n}(\pi) = \psi_n(\pi)_\red$ for $\pi\in H^i$ is the operator of cup product by $\psi_n(\pi)_\red\in H^{2n-2+i}(M_{r,d,D}^\ellipt)$.
\end{prop}
\begin{proof}
    Let $f$ be an operator. We have the unique decomposition
    \[
    f = \sum_{i,j} y^i f_{i,j} \partial_y^j,
    \]
    where $f_{i,j}$ commutes with both $y$ and $\partial_y$. Recall that $f_\red = f_{0,0}$. Applying $\Ad_y$, $\Ad_{\partial_y}$ to the above identity we obtain
    \[
    f_{i,j} = \left(\frac{1}{i! j!} \Ad_{\partial_y}^i (-\Ad_y)^j f\right)_\red.
    \]
    For any two operators $f, g$ we write
    \[
    f g = \sum_{i,j} y^i f_{i,j} \partial_y^j \sum_{i',j'} y^{i'} g_{i',j'} \partial_y^{j'}.
    \]
    The terms with $i>0$, $j'>0$ or $i'\neq j$ do not contribute to the constant term. The remaining terms with $i=j'=0$ and $i'=j$ contribute $j! f_{0,j} g_{j,0}$. Thus we have
    \[
    (f g)_\red = \sum_i i! f_{0,i} g_{i,0} = \sum_i \frac{1}{i!} \left((-\Ad_y)^i f\right)_\red (\Ad_{\partial_y}^i g)_\red.
    \]
    Applying this identity to each term in Corollary \ref{cor:rational reduced} and using
    \[
    [\partial_y, \wt D_{m,n}(\pi)] = n \wt D_{m,n-1}(\pi),\qquad [y, \wt D_{m,n}(\pi)] = \frac{m}{r} \wt D_{m-1,n}(\pi\omega)
    \]
    we obtain the required identities.
\end{proof}

We finish this section by relating the algebra $\Wred$ with cumbersome relations of Proposition~\ref{prop:rational reduced twice} to the simpler algebra $\CH_2$ of Hamiltonian vector fields on the plane. Although not directly used in the rest of the paper, it was one of our main motivations to study Hecke operators and their relation with tautological classes. The symmetry between $m$ and $n$ is evident in the relations \eqref{eq:ws red higgs} of $\Wred$. 
Let us introduce a formal variable $x$ and define the ``unreduced'' operators $\wt D_{m,n}(\pi)_\unred$ by
    \[
    \wt D_{m,n}(\pi)_\unred = \sum_{i,j} x^i \binom{m}{i} \binom{n}{j} (-r)^{-j} \wt D_{m-i,n-j}(\pi\omega^j) \partial_x^j.
    \]
    This produces an action of the Lie algebra with bracket
    \[
    [\wt D_{m,n}(\pi)_\unred, \wt D_{m',n'}(\pi')_\unred] = (m'n - m n') \wt D_{m+m'-1,n+n'-1}(\pi\pi')_\unred
    \]
    on the space $H^*_\taut(\CM_{r,d,D}^\ellipt)[x]$. By direct computation, the subalgebra spanned by elements $D_{m,n}(1)$ is isomorphic to the Lie algebra $\CH_2$ of polynomial Hamiltonian vector fields on the plane.

    \begin{cor}\label{cor:rational}
        The Lie algebra $\CH_2$ acts on
        \[
        H^*_\taut(\CM_{r,d,D}^\ellipt)[x] = H^*_\taut(M_{r,d,D}^\ellipt)[x,y]
        \]
        in such a way that $\psi_n(1)$ corresponds to the operator with Hamiltonian $y^n$.\qed
    \end{cor}

    Note that a similar action of $\CH_2$ on the Chow ring of the Jacobian of $C$ was obtained in~\cite[Thm. 0.1]{polishchuk}.

\subsection{Tautological classes in the parabolic case}
Our argument to prove the $P=W$ conjecture treats simultaneously the classical and parabolic cases.
Let us explain how to extend the constructions of Sections \ref{sec:hecke}-\ref{sec:sl2} to the parabolic setup.
Let $\CM_{r,d,D}^\parell$ resp. $M_{r,d,D}^\parell$ be the moduli stack resp. coarse moduli space parameterizing $D$-twisted Higgs bundles of rank $r$ and degree $d$ such that the spectral curve is reduced and irreducible, the residue of the Higgs field over each point of $D$ has distinct eigenvalues, together with an ordering of the eigenvalues at each point.
Equivalently, one can view this data as a purely 1-dimensional sheaf $\CF$ on $S^\circ$ whose support $\supp(\CF)$ is proper, reduced, irreducible, and intersects the fibers $\pi^{-1}(p)$ transversally for every $p \in D$ at $r$ points, together with an ordering of $\pi^{-1}(p) \cap \supp(\CF)$ for every $p\in D$.
We call it the \textit{parabolic elliptic} stack. There is an obvious forgetful map
\[
\mathsf{For}:\CM_{r,d,D}^\parell \to \CM_{r,d,D}^\ellipt,
\]
which is an unramified covering of degree $(r!)^{|D|}$ over its image.
As Hecke operators do not change the support, the results of the previous section apply to these moduli spaces by base change; in particular we obtain actions of the respective algebras on $H^*_\taut(\CM_{r,d,D}^\parell)$, $H^*_\taut(M_{r,d,D}^\parell)$, which are the subrings generated by the tautological classes of the form $\psi_k(\pi)$, $\psi_k(\pi)_\red$ respectively, for $\psi \in \Pi$. However, due to the presence of marked points we can introduce additional tautological classes and Hecke operators.

\smallskip

The tautological sheaf $\CF_{r,d,D}^\parell$ is isomorphic to the pullback $\mathsf{For}^*(\CF_{r,d,D})$ of the tautological sheaf on $\CM_{r,d,D}$, so we will omit the superscript for simplicity. 
For each $p\in D$, $i=1,\ldots,r$, denote by $q_{p,i}:\CM_{r,d,D}^\parell \to S^\circ\subset S$ the map sending a parabolic Higgs bundle to the $i$-th point of $\supp(\CF)\cap \pi^{-1}(p)$.
We have a tautological sheaf
\[
    \mathcal{U}^d_{p,i} \coloneqq (\Id\times q_{p,i})_*(\Id\times q_{p,i})^*\CF_{r,d,D}^\parell,
\]
whose restriction to $\{\CF\}\times S$, $\CF\in \CM_{r,d,D}^\parell$ is the skyscraper sheaf $\delta_{q_{p,i}}$.
Note that $(\Id\times q_{p,i})^*\CF_{r,d,D}^\parell$ is a line bundle.
In particular, we have
\[
    \Hom(\CF_{r,d,D}^\parell, \mathcal{U}^d_{p,i}) = \Hom((\Id\times q_{p,i})^*\CF_{r,d,D}^\parell,(\Id\times q_{p,i})^*\CF_{r,d,D}^\parell) = \mathbb{C}.
\]
As a consequence, there is a canonical subsheaf $\CG \subset \CF_{r,d,D}^\parell$ with $\CF_{r,d,D}^\parell/\CG \simeq \mathcal{U}^d_{p,i}$.
This gives rise to an isomorphism of stacks
$\Mod_{p,i}:\CM_{r,d,D}^\parell \to \CM_{r,d-1,D}^\parell$ and a canonical short exact sequence
\begin{equation}\label{eq:taut classes par ell Mod}
0 \xrightarrow{} \Mod_{p,i}^*(\CF_{r,d-1,D}) \xrightarrow{} \CF_{r,d,D} \xrightarrow{} \Hom_S(\CF_{r,d,D},\mathcal{U}^d_{p,i})^\vee\otimes \mathcal{U}^d_{p,i} \xrightarrow{} 0.
\end{equation}
Observe that $\Mod_{p,i}^*(\mathcal{U}^{d-1}_{p,i}) \simeq \mathcal{U}^d_{p,i}$.
We set
\[
X_{p,i} \coloneqq \Mod_{p,i}^*: H^*(\CM_{r,d-1,D}^\parell) \xrightarrow{\sim} H^*(\CM_{r,d,D}^\parell).
\]
Further, set $L^d_{p,i}\coloneqq \Hom_S(\CF_{r,d,D},\mathcal{U}^d_{p,i})^\vee$, which is a line bundle on $\CM_{r,d,D}^{\parell}$. 
In terms of Higgs bundles, the fiber of $L^d_{p,i}$ at a pair $(\CE,\theta)$ is the $i$-th eigenspace of the residue of the Higgs field $\theta$ at $p$.
We set 
$$y_{p,i}=c_1(L^d_{p,i})\in H^2(\CM_{r,d,D}^\parell)$$
and denote by the same symbol the operator of cup product with $c_1(L^d_{p,i})$.

\begin{prop}\label{prop:Hrd-rels}
    For each $p\in D$, $1\leq i,i'\leq r$, $n\geq 0$, $\pi\in H^{>0}$ we have
    \[
    [X_{p,i}, y_{p',i'}] = 0, \quad [\psi_n(1), X_{p,i}] = n y_{p,i}^{n-1} X_{p,i}, \quad [\psi_n(\pi), X_{p,i}] = 0, \quad [y_{p,i}, T_0(1)] = X_{p,i}.
    \]
\end{prop}
\begin{proof}
The first three relations follow easily from~\eqref{eq:taut classes par ell Mod}. The last relation is less straightforward. Recall that $T_0(1)$ is given by the Hecke correspondence $\CZ$ which parametrizes colength one inclusions $\CF'\subset \CF$ with $\CF \in \CM_{r,d,D}^\parell$, $\CF'\in \CM_{r,d-1,D}^\parell$. Denoting by $\pi_d, \pi_{d-1}$ the maps to $\CM_{r,d,D}^\parell, \CM_{r,d-1,D}^\parell$ we have
$$T_0(1)(c)=\pi_{d*}\left([\CZ]^\mathsf{vir} \cap \pi_{d-1}^*(c)\right).$$
It follows that
$$[y_{p,i},T_0(1)](c)=\pi_{d*}\left( c_1(\pi_{d}^*(L_{p,i}^d)-\pi^*_{d-1}(L_{p,i}^{d-1})) \cap \pi_{d-1}^*(c) \cap [\CZ]^{\mathsf{vir}}\right).$$
By construction, there is a canonical inclusion $\pi_{d-1}^*(L^{d-1}_{p,i}) \to \pi^*_{d}(L_{p,i}^d)$ whose cokernel is supported on the closed substack
$$\CZ_{p,i}\coloneqq \{\CF' \subset \CF\;|\; \supp(\CF/\CF')=\{q_{p,i}\}\} \subset \CZ.$$
Observe that $\pi_d$ restricts to an isomorphism $\CZ_{p,i} \simeq \CM_{r,d,D}^\parell$ under which $\pi_{d-1}$ gets identified with $\Mod_{p,i}$. Put $c_{p,i}=[y_{p,i},T_0(1)](1) \in H^0(\CM_{r,d,D}^\parell)=\BC$. By the projection formula we have $[y_{p,i},T_0(1)]=c_{p,i}X_{p,i}$, hence it only remains to determine the constant $c_{p,i}$. Note that $c_{p,i}$ is independent of $i$ by $\mathfrak{S}_r$-symmetry relabeling the points $\{q_{p,i}\}_i$. To determine this constant we note that $\sum_{i=1}^r y_{p,i}$ is the first Chern class of the rank $r$ bundle 
\begin{equation}\label{eq:sum pi n}
\CF_{r,d,D}/\CF_{r,d,D}(-\omega)=\bigoplus_i L_{p,i}^d
\end{equation}
i.e. $\sum_{i=1}^ry_{p,i}=\psi_1(\omega)$. It follows that
\begin{equation}
     \sum_{i=1}^p c_{p,i}X_{p,i} = [\psi_1(\omega), T_0(1)] = T_0(\omega).
\end{equation}
Evaluating this at $1$, we obtain $\sum_i c_{p,i} = r$ and thus $c_{p,i}=1$ as desired.
\end{proof}

\subsection{Algebras in the parabolic case}
Consider the following ring extensions of $H$ and $\overline{H}\coloneqq H^*(S)$.

\begin{defn}\label{def:Hrd}
    Let $\overline{H}_{r,D}$ be the ring extension of $\overline{H}$ obtained by adding generators $\tau_{p,i}$ of degree $2$ where $p\in D$, $i=1,\ldots,r$ modulo the following relations:
    \begin{equation}\label{eq:rel-Hrd}
    \sum_i \tau_{p,i} = \omega,\quad \tau_{p,i} H^{>0}=0,\quad \tau_{p,i} \tau_{q,j} = 0,\quad \tau_{p,i}\xi = \frac{1}{r}\omega\xi \quad (p,q\in D, 1 \leq i,j \leq r).
    \end{equation}
    We denote $H_{r,d} \coloneqq \overline{H}_{r,d}/(\xi)$; equivalently, $H_{r,d}$ is the ring extension of $H$ by generators $\tau_{p,i}$ satisfying the first three relations of~\eqref{eq:rel-Hrd}.
\end{defn}

Note that the diagonal class $\overline{\Delta}\in \overline{H}\otimes \overline{H}$, considered as an element of $\overline{H}_{r,D}\otimes \overline{H}_{r,D}$, satisfies all the conditions of Section~\ref{sec:Fock general}.
Hence as explained in Section~\ref{rem:any-H}, the algebraic study of $W$-algebras in Section~\ref{sec:fock space} applies to $H_{r,d}$.
In order to make the connection with the parabolic elliptic stack and apply the results of Section~\ref{sec:sl2}, it remains to verify the analogue of Corollary~\ref{cor:surface W}.
This amounts to defining $\psi_n(\tau_{p,i}), T_n(\tau_{p,i})$ in such a way that the formulas~(\ref{eq:hecke commutation}, \ref{eq:hecke segre}) continue to hold.
To this end, set for each $n\geq 0$
\[
\psi_n(\tau_{p,i}) \coloneqq y_{p,i}^n, \qquad T_n(\tau_{p,i}) \coloneqq y_{p,i}^n X_{p,i}.
\]
This assignment is compatible with the relations $\sum_i \tau_{p,i} = \omega$ because by~\eqref{eq:sum pi n}
\[
\sum_{i=1}^r y_{p,i}^n = \psi_n(\omega),
\]
and by Proposition~\ref{prop:Hrd-rels}
\[
\sum_{i=1}^r y_{p,i}^n X_{p,i} = \sum_{i=1}^r \frac{1}{n+1} [y_{p,i}^{n+1}, T_0(1)] = \frac{1}{n+1} [\psi_{n+1}(\omega), T_0(1)] = T_n(\omega).
\]
Continuing to use Proposition~\ref{prop:Hrd-rels} we have, for $\pi\in H$,
\begin{equation}\label{eq:compatibility parab H}
[y_{p,i}^n, T_{n'}(\pi)] = \frac{1}{n'+1} [y_{p,i}^n, [\psi_{n'+1}(\pi), T_0(1)]] = \frac{n}{n'+1} y_{p,i}^{n-1} [\psi_{n'+1}(\pi), X_{p,i}].
\end{equation}
The r.h.s. of \eqref{eq:compatibility parab H} is equal to $n y_{p,i}^{n+n'-1} X_{p,i} = n T_{n+n'-1}(\tau_{p,i})$ for $\pi = 1$ and vanishes for $\pi\in H^{>0}$. In addition, we have
\[
[\psi_n(1), y_{p,i}^{n'} X_{p,i}] = n y_{p,i}^{n+n'-1} X_{p,i} = n T_{n+n'-1}(\tau_{p,i}).
\]
All together, we see that the commutation relations
\[
[\psi_n(\pi), T_{n'}(\pi')] = n T_{n+n'-1}(\pi\pi')
\]
hold for all $\pi,\pi'\in H_{r,D}$, which verifies~\eqref{eq:hecke commutation}.
From the definition of $T_n(\tau_{p,i})$ we have
\[
T_n(\tau_{p,i})(1) = y_{p,i}^n.
\]
Note that the diagonal class of the surface $S^\circ$ equals to $\omega\otimes\omega$, and therefore annihilates $\tau_{p,i}$.
This implies that for any two symmetric functions $f$, $g$ of positive degree we have $(fg)(\tau_{p,i}) = 0$. Hence
\[
h_{n+1}(\tau_{p,i}) = \frac{p_{n+1}(\tau_{p,i})}{n+1} = \psi_n(\tau_{p,i}) = y_{p,i}^n,
\]
and so the formula~\eqref{eq:hecke segre} also holds.
This finishes the proof of Corollary~\ref{cor:surface W} for $H_{r,d}$.
Given this, we run the same arguments as in Section~\ref{sec:sl2} to deduce the following result.

\begin{thm}\label{thm:rational parabolic}
Corollaries \ref{cor:trigonometric}, \ref{cor:rational reduced}, \ref{cor:rational} and Proposition \ref{prop:rational reduced twice} extend \textit{verbatim} to the spaces $\CM_{r,d,D}^\parell$ and $M_{r,d,D}^\parell$ when $H$ is replaced by $H_{r,D}$ and the tautological ring is enlarged by adding the classes $y_{p,i}$.
\end{thm}

From now on, we denote by $H^*_\taut(\CM_{r,d,D}^\parell)\subset H^*(\CM_{r,d,D}^\parell)$,  $H^*_\taut(M_{r,d,D}^\parell)\subset H^*(M_{r,d,D}^\parell)$ the subrings generated by $\psi_n(\pi)$, $\pi\in H_{r,D}$.

\subsection{$\liesl_2$-triples}\label{sec: sl 2 triples}
In Section~\ref{ssec:sl2} we explicitly constructed an $\liesl_2$-triple acting on $H^*_\taut(M_\alpha)$ from the algebra $\Wred$. As explained in Remark \ref{rem:undeformed}, the situation simplifies in the case of Higgs bundles on a curve. Let us spell out the form of the $\liesl_2$-triple in our case. As this applies both to the parabolic and the classical cases, we write $\CM$, resp. $M$ for either $\CM_{r,d,D}^\parell$ or $\CM_{r,d}$, resp. $M_{r,d,D}^\parell$ or $M_{r,d}$.

Recall that by construction $\wt D_{1,0}(\eta)=\wt D_{1,0}(\omega)$ vanishes. Hence $\wt D_{1,0}(\omega)_\red=0$. Likewise, we have $y_\red=0$ and $(\partial_y)_\red=0$. The element $D_{0,1}(1) = \psi_1(1)\in H^0(\CM)$ being the Euler characteristic of a vector bundle of rank $r$ and degree $d$, we may fix $D_{0,1}(1)=0$ by choosing $d$ appropriately\footnote{Recall that the operators $\wt D_{m,n}(\xi)_\red$ act on the cohomology $H^*(M_\alpha)$ for each $\alpha \in \alpha_0 + \BZ\delta$, and that these have all been identified; in particular, we may freely choose the degree $d$. By Remark~\ref{rmk:ci-from-rd}, our choice of $d$ is a multiple of $r$.}. Thus,
\[
\wt D_{0,1}(1)_\red = \wt D_{1,0}(1)_\red = \wt D_{0,1}(\omega)_\red = \wt D_{1,0}(\omega)_\red=0
\]
and $\wt D_{0,2}(1)=\psi_2(1)$. This implies that all the extra terms in Proposition \ref{prop:rational reduced twice} vanish for the operators $\wt D_{1,1}(1)_\red$, $\wt D_{0,2}(1)_\red$, $\wt D_{2,0}(1)_\red$, so that we have
\[
[\wt D_{2,0}(1)_\red, \wt D_{m,n}(\pi)_\red] = -2n \wt D_{m+1,n-1}(\pi)_\red,
\]
\[
[\wt D_{1,1}(1)_\red, \wt D_{m,n}(\pi)_\red] = (m-n) \wt D_{m,n}(\pi)_\red,
\]
\[
[\wt D_{0,2}(1)_\red, \wt D_{m,n}(\pi)_\red] = 2m \wt D_{m-1,n+1}(\pi)_\red.
\]

\begin{prop}\label{prop:sl2}
    The operators $\frake = \frac12 \wt D_{0,2}(1)_\red$, $\frakh = -\wt D_{1,1}(1)_\red$, $\frakf = -\frac12 \wt D_{2,0}(1)_\red$ form an $\liesl_2$-triple acting on $H_\taut^*(M)$. \qed
\end{prop}
We will need a slightly more general statement. The $\liesl_2$-triple constructed above will be called the \emph{original} $\liesl_2$-triple. The space $H_\taut^2(M)$ is spanned by $\psi_2(1)$, $\psi_1(\pi)$ for $\pi\in H^2_{r,D}$ and the products $\psi_1(\gamma_i)\psi_1(\gamma_j)$ for $1 \leq i,j \leq 2g$ (recall that $\gamma_i$'s form a basis of $H^1$).
\begin{prop}\label{prop:many sl_2}
    Fix $\alpha\in H^2_\taut(M)$ and write
    \begin{equation}\label{eq:proof variation sl 2}
    \alpha = A\left(\psi_2(1) + \frac{1}r \sum_{i=1}^g \psi_1(\gamma_i)\psi_1(\gamma_{i+g})\right) + \sum_{i<j} B_{i,j} \psi_1(\gamma_i)\psi_1(\gamma_j) + \sum_{p,i} C_{p,i} \psi_1(\tau_{p,i}),
    \end{equation}
    for some $A,B_{i,j},C_{p,i}\in\BC$, where the third term only occurs in the parabolic case. If $A\neq 0$ and the antisymmetric matrix with off-diagonal entries $B_{i,j}$ is non-degenerate, then there exists an $\liesl_2$-triple $(\frake',\frakh',\frakf')$ acting on $H_\taut^*(M)$ with $\frake'=\alpha$, whose associated $\frakh'$-degree filtration coincides with the $\frakh$-degree filtration of the original $\liesl_2$-triple.
\end{prop}
\begin{proof}
Consider the operators $y(\pi) = \psi_1(\pi)$ and $x(\pi) = \wt D_{1,0}(\pi)_\red$ for $\pi\in H^{>0}_{r,D}$. By Theorem~\ref{thm:sl2-general}, the operators $y(\pi)$ increase the $\frakh$-degree by $1$, and the operators $x(\pi)$ decrease it by $1$. In particular, they are both nilpotent as $H^*(M)$ is finite-dimensional. Proposition~\ref{prop:rational reduced twice} yields
\begin{equation}\label{eq:y-x-comm}
    [y(\pi), x(\pi')] = \psi_0(\pi \pi').
\end{equation}
Note that $\psi_0(\pi \pi')=r \pi\cdot \pi'$ if both $\pi,\pi'$ are in $H^1$ (and the intersection product is taken in $H^*(C)$) while $\psi_0(\pi \pi')=0$ if one of $\pi,\pi'$ belongs to $H_{r,D}^{>1}$. In addition we have
    \[
    [\frake, x(\pi)] = y(\pi).
    \]
 For $\pi\in H^2_{r,D}$ let $X(\pi)=e^{x(\pi)}$.
 By the equation above it satisfies $[\frake, X(\pi)] = y(\pi) X(\pi)$, and so
    \[
    X(\pi)^{-1} \frake X(\pi) = \frake + y(\pi).
    \]
Starting from the original $\liesl_2$ triple and conjugating by a product of complex powers of $X(\tau_{p,i})$ (as $x(\tau_{p,i})$ is nilpotent, the complex powers of $X(\tau_{p,i})$ are well-defined) we can obtain a new $\liesl_2$-triple $(\frake',\frakh',\frakf')$ in which $\frake'=\frake + \sum_{p,i} C_{p,i} \psi_1(\tau_{p,i})$ for any given complex values of the coefficients $C_{p,i}$. Moreover, since $x(\pi)$ has negative $\frakh$-degree, the operators $X(\tau_{p,i})$ preserve the $\frakh$-degree filtration, i.e. the $\frakh$-degree and $\frakh'$-degree filtrations coincide.

In order to obtain the general form, we again argue in a similar fashion, this time using operators of the form $x(\gamma) y(\gamma')$ for $\gamma, \gamma'\in H^1$. Notice that these operators preserve the $\frakh$-degree and hence the $\frakh$-degree (and $\frakh'$-degree) filtration. A direct computation yields
    \[
    [\frake, x(\gamma) y(\gamma')] = y(\gamma) y(\gamma'), \qquad\left[\sum_i y(\gamma_i) y(\gamma_{i+g}), x(\gamma) y(\gamma')\right] = - r y(\gamma) y(\gamma').
    \]
As $[x(\gamma)y(\gamma'),y(\tau_{p,i})]=0$ for any $\gamma,\gamma',p,i$, the element $\frake'+\frac{1}{r} \sum_i y(\gamma_i) y(\gamma_{i+g})$ commutes with any element of the form $x(\gamma) y(\gamma')$. Equation~\eqref{eq:y-x-comm} implies that the elements of the form $x(\gamma) y(\gamma')$ generate a Lie algebra isomorphic to $\liegl_{2g}$. The adjoint action by this Lie algebra can be integrated to an action of $\GL_{2g}(\BC)$ on $H^2_\taut(M)$ which fixes elements of the form $\frake'+\frac{1}{r} \sum_i y(\gamma_i)y(\gamma_{i+g})$ and transforms elements of the form $y(\gamma) y(\gamma')$ according to the exterior square of the standard representation. Since any two non-degenerate antisymmetric forms are $\GL_{2g}(\BC)$-equivalent, the orbit of $-\sum_{i\leq g}y(\gamma_i)y(\gamma_{i+g})$ contains all the nondegenerate elements. Writing $\frake''=(\frake'+\frac{1}{r} \sum_i y(\gamma_i) y(\gamma_{i+g}))-\frac{1}{r} \sum_i y(\gamma_i) y(\gamma_{i+g})$ we see that we can transform the $\liesl_2$-triple $(\frake',\frakh',\frakf')$ into another $\liesl_2$-triple of the desired form. We conclude by noting that the $\GL_{2g}(\BC)$-action preserves the $\frakh$-degree.
\end{proof}

\begin{rem}
    It is clear from the proof that the symplectic group $\Sp_{2g}(\BC)\subset\GL_{2g}(\BC)$ acts on $H_\taut^*(M)$ respecting the $\liesl_2$-action, preserving the cohomological degree, and acting via the standard representation on the span of classes $\psi_1(\gamma_i)$.
\end{rem}

\begin{cor}\label{cor:generic sl_2}
    For any $\alpha\in H^2_\taut(M)$ and for all but finitely many values $\lambda\in\BC$ there exists an $\liesl_2$-triple $(\frake,\frakh,\frakf)$ with $\frake=\alpha+\lambda \psi_2(1)$ for which the $\frakh$-degree filtration coincides with the $\frakh$-degree filtration of the original $\liesl_2$-triple.\qed
\end{cor}

\section{Proof of the P=W conjectures}\label{sec:PW}
Throughout this section, we fix a smooth complex projective curve $C$ of genus $g>1$ and a positive integer $r$. We will consider Higgs bundles on $C$ of rank $r$ and degree $d$; the latter will be allowed to vary. We also fix a (nontrivial) simple effective divisor $D$ which we specialize in Section~\ref{sec:classical case} to a divisor of degree one.

\subsection{Strategy of the proof}\label{sec:PW1} Let us briefly explain the steps of our proof of the $P=W$ conjectures, in both classical and parabolic settings. First of all, thanks to a classical result of Markman (\cite{markman2002generators}, see \cite{parabolicmarkman} for the parabolic case), the cohomology of the moduli spaces of stable Higgs bundles is generated by the tautological classes. The $P=W$ conjectures then reduce, by work of Shende \cite{shende2016weights} (see \cite{mellit2019cell} for the parabolic case) to the $P=C$ conjecture, see Introduction for the statement. 
Our proof of the $P=C$ conjecture proceeds in four steps, each one establishing a version of $P=C$ in a different setup. 

Recall that a $D$-parabolic Higgs bundle\footnote{this is sometimes referred to as a \textit{twisted} parabolic Higgs bundle.} is a triple $(\CE,\theta,(F_p)_p)$ where $\CE$ is a vector bundle on $C$, $\theta: \CE \to \CE \otimes \Omega(D)$ and for each $p \in D$, $F_p: F_{p,1} \subset F_{p,2} \subset \cdots \subset F_{p,r}=\CE_{|p}$ is a full flag in the fiber of $\CE$ at $p$ which is preserved by $\text{res}_p(\theta)$. A $\textit{nilpotent}$ $D$-parabolic Higgs bundle\footnote{this is sometimes simply called a parabolic or a strongly parabolic Higgs bundle.} is a triple $(\CE,\theta,(F_p)_p)$ as above, for which $\text{res}_p(F_{p,i}) \subseteq F_{p,i-1}$ for all $i$. Finally, note that a parabolic Higgs bundle $(\CE, \theta, (F_p)_p)$ with $\text{res}_p=0$ for all $p$ is simply a Higgs bundle $(\CE,\theta)$ together with a collection of full flags in each of the fibers $\CE_{|p}$. 

We proceed as follows:
\begin{enumerate}
\item Identify the $\frakh$-degree filtration on $H^*(M^\parell_{r,d,D})$ from Section~\ref{sec: sl 2 triples} with the perverse filtration $P$, and deduce $P=C$ for the parabolic elliptic moduli $M^\parell_{r,d,D}$ from commutation relations in $\Wred$ (Section~\ref{ssec:elliptic});
\item Check that the restriction map $H^*(M^\parab_{r,d,D})\to H^*(M^\parell_{r,d,D})$ is injective and compatible with $P$, deducing $P=C$ for stable parabolic Higgs bundles (Section~\ref{ssec:twisted parabolic});
\item Check that the restriction map $H^*(M^\parab_{r,d,D})\to H^*(M^{\parab,0}_{r,d,D})$ to the nilpotent parabolic Higgs bundles is an isomorphism compatible with $P$, hence $P=C$ for $M^{\parab,0}_{r,d,D}$ (Section~\ref{ssec:nil-par});
\item Finally, for the classical moduli of Higgs bundles $M_{r,d}$, $\gcd(r,d)=1$ we realize $H^*(M_{r,d})$ as a direct summand of $H^*(M^{\parab,0}_{r,d,D})$, deducing $P=C$ by another compatibility check (Section~\ref{sec:classical case}).
\end{enumerate}

\begin{rem}
    Note that we only use the $W_S$-action and its consequences in the first step. The reason is that our construction \textit{a priori} does not give an action of $W_S$ on $H^*(M^\parab_{r,d,D})$ and other moduli spaces.
\end{rem}

\subsection{Elliptic case}\label{ssec:elliptic}
In an effort to unburden the notation, we will suppress the subscript $D$. Fix a generic stability parameter and let $M^\parab_{r,d}$ be the moduli space of stable $D$-parabolic Higgs bundles. Let $\chi^\parab: M^\parab_{r,d} \to \BA^\parab$ be the Hitchin map; here $\BA^\parab$ is an affine space which only depends on $r, D$ and $C$.
There is a cartesian square
$$
\begin{tikzcd}
    M_{r,d}^\parell \arrow{r}{i} \arrow[swap]{d}{\chi^\parell} & M_{r,d}^\parab \arrow{d}{\chi^\parab}\\
    \BA^\parell \arrow{r}{i'} & \BA^\parab
\end{tikzcd}
$$
where $i':\BA^\parell \to \BA^\parab$ and $i$ are open immersions ($\BA^\parell$ is the parabolic elliptic locus). The map $\chi^{\parell}$ is projective and hence yields a perverse filtration $P_\bullet$ on $H^*(M_{r,d}^\parell)$.

By~\cite[Thm. 1.2, Cor. 1.3]{parabolicmarkman} or~\cite[Thm. 3.1.8]{oy2017parabolic}, we have the following parabolic version of Markman's theorem:
\begin{equation}\label{eq:pure = tauto}
H^*_\taut(M_{r,d}^\parab) = H^*_\pure(M^\parab_{r,d}),
\end{equation}
from which we deduce the same result for $M_{r,d}^\parell$ by restricting along $i$.

As Hecke correspondences do not change the support of purely $1$-dimensional sheaves, the convolution diagram \eqref{diag:Hecke corr} is defined over $\BA^\parab$. Applying Propositions \ref{prop:perverse cup product} and \ref{prop:functoriality} we deduce that for $\xi\in H^i(S)$ and $m\geq 0$
\begin{itemize}
    \item $\psi_m(\xi) P_j \subseteq P_{j+2m-2+i}$;
    \item $T_m(\xi) P_j \subseteq P_{j+2m-2+i}$;
    \item More generally, the operator $D_{m,n}(\xi)$  (of vertical degree $2n-2+i$) satisfies  $D_{m,n}(\xi) P_j \subseteq P_{j+2n-2+i}$.
\end{itemize}
In particular, the operator $q_1(\omega)$, which was used to identify the cohomologies of $M^\parell_{r,d}$ for different values of $d$ respects the filtration, i.e. $q_1(\omega) P_iH^*(M^\parell_{r,d})\subseteq P_iH^*(M^\parell_{r,d+1})$. Tensoring by a line bundle of degree one yields an isomorphism $M^\parell_{r,d} \simeq M^\parell_{r,d+r}$ which is compatible with the Hitchin map; it follows that $H^*(M^\parell_{r,d})$ and $H^*(M^\parell_{r,d+r})$ have the same perverse filtrations; applying $q_1(\omega)^r$, we easily deduce that $q_1(\omega)P_iH^*(M^\parell_{r,d})= P_iH^*(M^\parell_{r,d+1})$ for any $d$ (hence the same holds for $q_1(\omega)^{-1}$).
Since the operators $\wt D_{m,n}(\xi)_\red$ have the same vertical degree as $D_{m,n}(\xi)$, we obtain the following:

\begin{prop}\label{prop:compatibility sl 2 perverse}
    The operators $\frake, \frakh, \frakf$ of the original $\liesl_2$-triple (see Proposition~\ref{prop:sl2}, as well as any of the triples constructed in Proposition~\ref{prop:many sl_2} change perversity by $2,0,-2$ respectively.
\end{prop}

This completely pins down the perverse filtration:
\begin{prop}\label{prop:P and sl2}
    The perverse filtration on $H^*_\pure(M^\parell_{r,d})$ coincides with the $\frakh$-degree filtration induced by the original $\liesl_2$-triple.
\end{prop}
\begin{proof}
Choose a relatively ample class $\alpha\in H^2(M^\parell_{r,d})$ and replace it by some linear combination $\alpha'=\alpha + \lambda \psi_2(1)$ with $\lambda\in\BQ$ which is relatively ample and for which the operator $\frake$ of cup product with $\alpha'$ is a part of an $\liesl_2$-triple $(\frake,\frakh',\frakf')$ (see Corollary \ref{cor:generic sl_2}). By Proposition~\ref{prop:compatibility sl 2 perverse}, $(\frake,\frakh',\frakf')$ descends to an $\liesl_2$-triple of operators on $\mathrm{Gr}^P_\bullet H^*(M^\parell_{r,d})$, for which $\frakh'$ is of degree zero. By Theorem~\ref{thm:decomposition}, one can complete the operator $\frake$ to an $\liesl_2$-triple $(\frake,\frakh,\frakf)$ acting on $\mathrm{Gr}^P_\bullet H^*(M^\parell_{r,d})$ for which $\frakh_{|\mathrm{Gr}^P_i}=i\Id_{\mathrm{Gr}^P_i}$. In particular, $[\frakh,\frakh']=0$. The elementary lemma below allows us to conclude that $\frakh=\frakh'$, from which Proposition~\ref{prop:P and sl2} readily follows. 
\end{proof}
\begin{lem} Let $V$ be a finite-dimensional vector space and let $(\frake,\frakh,\frakf), (\frake,\frakh',\frakf')$ be two $\liesl_2$-triples in $\End(V)$. Assume that $[\frakh,\frakh']=0$. Then $\frakh=\frakh'$.
\end{lem}
\begin{proof} By the Jacobson-Morozov theorem, there exists $g \in \text{Stab}_{GL(V)}(\frake)$ such that $\Ad_g\frakh=\frakh'$. Write $V=\bigoplus_i V[i]$ for the $\frakh$-weight decomposition and $V^\frakf=\bigoplus_i V^\frakf[i]$ for the space of lowest weight vectors. Then $g$ is determined by its action on each $V^\frakf[i]$ and we have
$$g(V^\frakf[i]) \subseteq (U(\frake) \cdot V^\frakf[>i]) \oplus V^\frakf[i].$$
This reflects the decomposition $\text{Stab}_{GL(V)}(\frake) \simeq \prod_i GL(V^\frakf[i]) \times \text{Rad}(\text{Stab}_{GL(V)}(\frake))$. As $[\frakh,\frakh']=0$, $\frakh'$ preserves each $V[i]$. But this is only possible if $g \in \prod_i GL(V^\frakf[i])$, which implies that 
$g$ commutes with $\frakh$ and thus $\frakh'=\frakh$.
\end{proof}

The following statement is the version of $P=C$ for the parabolic elliptic locus:
\begin{thm}\label{prop:PW elliptic}
    The subspace $P_m H^*_\pure(M_{r,d}^\parell)$ is the span of products $\prod_i \psi_{m_i}(\xi_i)$ satisfying $\sum_i m_i\leq m+N$, where $N=(g-1)r^2+ |D|\binom{r}{2} +1$.
\end{thm}
\begin{proof}
Observe that $1 \in H^0_\pure(M_{r,d}^\parell)$ is of perverse degree equal to the dimension of the Hitchin map, which is equal to $(g-1)r^2+|D|\binom{r}{2} +1$. The statement is now an obvious consequence of the relation $[\frakh, \psi_m(\xi)] = m \psi_m(\xi)$.
\end{proof}

\subsection{Parabolic case}\label{ssec:twisted parabolic}
Next, we consider the entire space $M^\parab_{r,d}$. As the stability parameter considered is primitive, there are no strictly semistables and the Hitchin map $\chi^\parab$ is projective. The group $\BC^*$ acts on $M^\parab_{r,d}$ by scaling the Higgs field, turning it into a semi-projective variety, see \cite{hausel2015cohomology}. The cohomology of $M^\parab_{r,d}$ is pure by~\cite[Corollary~1.3.2]{hausel2015cohomology}, hence~\eqref{eq:pure = tauto} yields
\[
H^*_\taut(M^\parab_{r,d}) = H^*(M^\parab_{r,d}).
\]

There is a decomposition $\BA^\parab = \BA_D \times \BA^{\parab,0}$, and correspondingly $\chi^\parab = \chi_D \times \chi'$, where $\BA^{\parab,0}$ is the Hitchin base for the $D$-twisted Higgs bundles, $\BA_D\coloneqq\BC^{r |D|-1}$ and
\[
\chi_D:M^\parab_{r,d} \to \BA_D
\]
is the map sending $(\CE, \theta, F)$ to the ordered collection of the eigenvalues of $\res_p \theta$ for all $p\in D$. Note that the sum of all the eigenvalues is always zero, hence the rank of $\BA_D$.

The map $\chi_D$ is smooth, i.e. it is a surjective submersion. In fact, $M^\parab_{r,d}$ is known to be a Poisson variety and the fibers of $\chi_D$ are precisely the symplectic leaves. Moreover, the image of the tangent map at every point is dual to the kernel of the Poisson tensor, so the dimension of the image of the tangent map is constant. From Corollary 1.3.3 in \cite{hausel2015cohomology} we deduce
\begin{prop}\label{prop:restriction iso}
For each $\Bmu\in \BA_D$ the restriction map $H^*(M^\parab_{r,d})\to H^*(\chi_D^{-1}(\Bmu))$ is an isomorphism.
\end{prop}
In particular, the cohomology of all the fibers of $\chi_D$ are pure.
\begin{prop}\label{prop:elliptic injective}
    The restriction map $i^*:H^*(M^\parab_{r,d}) \to H^*(M^\parell_{r,d})$ is injective, i.e. $H^*( M^\parab_{r,d})=H^*_\pure(M^\parell_{r,d})$.
\end{prop}
\begin{proof}
For a generic (in the sense of~\cite[Def.~2.1.1]{hausel2011arithmetic}) $\Bmu\in \BA_D$, the identity
    \[
    \sum_{p} \res_p \trace(\theta)=0
    \]
implies that any parabolic Higgs bundle in $(\CE,\theta, (F_p)_p) \in \chi_D^{-1}(\Bmu)$ is simple, which in turn implies that the spectral curve of $(\CE,\theta)$ is reduced and irreducible, hence $\chi_D^{-1}(\Bmu)\subset M^\parell_{r,d}$. The composition
    \[
    H^*(M^\parab_{r,d}) \xrightarrow{i^*} H^*(M^\parell_{r,d}) \to H^*(\chi_D^{-1}(\Bmu))
    \]
    is an isomorphism by Proposition~\ref{prop:restriction iso}, and so $i^*$ is injective.
\end{proof}

We deduce from~\eqref{eq:pure = tauto} that $i^*$ induces an isomorphism of tautological rings $H^*(M^\parab_{r,d}) \simeq H^*_\taut(M^\parell_{r,d})$. By Theorem~\ref{thm:decomposition}, $i^*$ preserves the perverse filtrations and becomes a map of Lefschetz structures once a relatively ample class $\alpha \in H^2(M^\parab_{r,d})$ has been chosen. It therefore has to be an isomorphism of Lefschetz structures on $H^*(M^\parab_{r,d}) \simeq H^*_\taut(M^\parell_{r,d})$. Theorem~\ref{prop:PW elliptic} implies 

\begin{cor} 
   The $P=C$ conjecture holds for $H^*(M^\parab_{r,d})$, i.e. $P_mH^*(M^\parab_{r,d})$ is the span of products of tautological classes $\prod_i \psi_{m_i}(\xi_i)$ satisfying $\sum_i m_i\leq m+N$, where $N=(g-1)r^2+ |D|\binom{r}{2} +1$.
\end{cor}

\subsection{Nilpotent parabolic case}\label{ssec:nil-par}
Let us now denote by $M_{r,d}^{\parab,0}\coloneqq\chi_D^{-1}(0)\subset M^\parab_{r,d}$ the moduli space parameterizing stable nilpotent parabolic Higgs bundles. The restriction of the Hitchin morphism to $M_{r,d}^{\parab,0}$ induces the perverse filtration on $H^*(M_{r,d}^{\parab,0})$. Recall that by Proposition \ref{prop:restriction iso}, the restriction map $H^*(M^\parab_{r,d})\to H^*(M_{r,d}^{\parab,0})$ is an isomorphism.
\begin{prop}\label{prop:par-to-par0}
The restriction isomorphism $H^*(M^{\parab}_{r,d})\xrightarrow{\sim} H^*(M_{r,d}^{\parab,0})$ identifies the perverse filtrations on two sides.
\end{prop}
\begin{proof}
    For simplicity, let us write $X=M_{r,d}^{\parab}$, $X_0 = M_{r,d}^{\parab,0}$, $N=\dim X$, $N_0=\dim X_0$, and $n=N-N_0$. We have $X_0=\chi^{-1} (\BA_0)$ for $\BA_0=\{0\} \times \BA^{\parab,0}\subset \BA^\parab$ a linear subspace of codimension $n$.
    
    Let $\ol X$ be the projectivization with respect to the $\BC^*$-action on $X$ (see \cite{hausel1998compactification}):
    \[
    \ol X = \left(X \times \BC \setminus \chi^{-1}(0)\times \{0\}\right) / \BC^*,
    \]
    and let $\partial X = \ol X\setminus X$. The spaces $\ol X$ and $\partial X$ are projective. They are not necessarily smooth, but the singularities are finite quotient singularities, so they are rationally smooth. In particular, their cohomology is pure, and may be identified with their Borel-Moore homology. We define $\ol X_0$ in an analogous fashion. The Hitchin map extends to a projective morphism $\ol \chi:\ol X\to \ol \BA$, where $\ol \BA\coloneqq \left(\BA \times \BC \setminus \{0\}\right)/ \BC^*$ is the corresponding weighted projective space. This map restricts to $\ol \chi_0: \ol X_0 \to \ol \BA_0$, where $\ol \BA_0 \subset \ol \BA$ is the projective subspace associated to $\BA_0$. Let $L\in H^2(\ol X)$, resp. $L_0 \in H^2(\ol X_0)$, be the pullback of the hyperplane class from $\ol \BA$.  
     
    \begin{lem}We have
    \[
    H^*(X) = H^*(\ol X)/L H^*(\ol X), \qquad H^*(X_0) = H^*(\ol X_0)/L_0 H^*(\ol X_0)
    \]
    \end{lem}
    \begin{proof}
    Let $U= \ol X\setminus\chi^{-1}(0)$. Since $U$ is the total space of a line bundle over $\partial X$, its cohomology is pure, so the restriction map
    \[
    H^*(\ol X) \to H^*(U) = H^*(\partial X)
    \]
    is surjective. The kernel of this map is identified with the homology of $\chi^{-1}(0)$, and therefore $H_*(\chi^{-1}(0))$. Since $\chi$ is proper, $H^*(\chi^{-1}(0))$ is pure by duality. The restriction map $H^*(X)\to H^*(\chi^{-1}(0))$ is an isomorphism because the $\BC^*$-action retracts $X$ onto $\chi^{-1}(0)$. Hence $H^*(X)$ is pure, which implies the surjectivity of the restriction map $H^*(\ol X)\to H^*(X)$. By purity,  the long exact sequence in Borel-Moore homology for $X\subset \ol X$ yields the collection of short exact sequences
    \[
    0\to H^{*-2}(\partial X) \to H^*(\ol X) \to H^*(X) \to 0.
    \]
    The composition $H^*(\ol X) \to H^*(\partial X) \to H^{*+2}(\ol X)$ is the multiplication by the fundamental class $[\partial X]=L$. We obtain a short exact sequence
    \[
    H^{*-2}(\ol X)\xrightarrow{L} H^{*}(\ol X) \to H^*(X) \to 0,
    \]
    as claimed. By dualizing this sequence and using the Poincar\'e duality we get a longer exact sequence:
    \[
    0\to H_{2N-*+2}(X) \to H^{*-2}(\ol X)\xrightarrow{L} H^{*}(\ol X) \to H^*(X) \to 0.
    \]
Finally, the same arguments apply for $X_0, \ol X_0$. 
\end{proof}

Let $\iota:\ol X_0\to \ol X$ be the closed embedding. The operators $\iota_*, \iota^*$ both commute with the multiplication by $L$, i.e. $\iota^* L =L_0 \iota^*$ and $\iota_* L_0=L \iota_*$. The composition $\iota_*\iota^*$ is the operator of multiplication by $[\ol X_0]\in H^{2n}(\ol X)$. On the other hand, since $\ol X_0 = \ol\chi^{-1} (\ol \BA_0)$ and $\ol \BA_0$ is a projective subspace, we have $L^n = c [\ol X_0]$ where $c>0$ is the multiplicity. So we have
    \[
    \iota_* \iota^* = c^{-1} L^n.
    \]

    We claim that $\image \iota_* \subset \image L^n$. To see this, let $\alpha = \iota_* \beta$, where $\beta\in H^*(\ol X_0)$. Using the isomorphism
    \[
    \iota^*:H^*(\ol X)/L H^*(\ol X) \to H^*(\ol X_0)/L_0 H^*(\ol X_0),
    \]
    we can lift $\beta$ and write $\beta = \iota^* \beta' + L_0 \beta''$. This implies
    \[
    \alpha = c^{-1} L^n \beta' + L \iota_*\beta''.
    \]
    Let $\alpha' = \iota_*\beta''$. Repeating the argument for $\alpha'$ and continuing in this fashion $n$ times we obtain $\alpha\in L^n H^*(\ol X)$.

    Next, we claim that for any $i\geq 0$ we have $(\iota^*)^{-1} (\kernel L_0^{i} + \image L_0)\subset \kernel L^{i+n} + \image L$. Indeed, suppose $\alpha\in H^*(\ol X)$ is such that
    \[
    \iota^* \alpha = \beta + L_0 \beta',\qquad L_0^{i}\beta = 0.
    \]
    Applying $\iota_*$ and using the previous statement we obtain
   $ c^{-1} L^n \alpha = \iota_* \beta + L \iota_*\beta'$
    hence
   $ c^{-1} L^n \alpha = \iota_*\beta + L^{n+1} \beta''$
    and
    \[
    L^{i+n}(\alpha - c L\beta'') = c \iota_* L^i \beta = 0,
    \]
    which shows that $\alpha\in \kernel L^{i+n} + \image L$.

    Now suppose $\alpha\in P_i H^j(X_0)$. By Theorem~\ref{thm:decomposition} this means that $\alpha$ can be represented in $H^*(\ol X_0)$ by an element of $W_{i+{N_0-j}}\subset \kernel L^{i+N_0-j+1} + \image L$. By the above, inside $H^*(\ol X)$ the same element $\alpha$ can be represented by an element of $\kernel L^{i+N-j+1}$, which implies $\alpha\in P_i H^j(X)$. Thus, the inverse of the pullback map $\iota^*:H^*(X_0)\to H^*(X)$ preserves the perverse filtration. Choosing a relatively ample line class $\omega \in H^2(X)$, we obtain a map of Lefschetz structures, which is an isomorphism of vector spaces, hence also an isomorphism of Lefschetz structures.
\end{proof}

\begin{cor}\label{cor:parabolic}
    The $P=C$ conjecture holds for $H^*(M_{r,d}^{\parab,0})$, i.e. $P_mH^*(M^{\parab,0}_{r,d})$ is the span of products of tautological classes $\prod_i \psi_{m_i}(\xi_i)$ satisfying $\sum_i m_i\leq m+N$, where $N=(g-1)r^2+ |D|\binom{r}{2} +1$.
\end{cor}

\begin{rem}
From $H^*(X)\cong H^*(\chi^{-1}(0))\cong H^*(X_0)$ we obtain another proof of the fact that the restriction $H^*(M^{\parab}_{r,d})\to H^*(M^{\parab,0}_{r,d})$ is an isomorphism.
\end{rem}

\subsection{The classical case}\label{sec:classical case}
We now make the assumption that $\text{gcd}(r,d)=1$. Let $M_{r,d}$ be the moduli space of stable Higgs bundles of rank $r$ and degree $d$. Choose a point $p\in C$ and let $D=(p)$. Let $M_{r,d}^{\parab,0}$ be the moduli space of stable nilpotent $D$-parabolic Higgs bundles $(\CE,\theta, (F_p))$, with respect to a generic stability parameter. Let $\iota:\wt M^{\parab,0}_{r,d}\hookrightarrow M^{\parab,0}_{r,d}$ be the closed subvariety defined by the condition $\res_p\theta=0$. Then we have the forgetful morphism $\pi:\wt M^{\parab,0}_{r,d}\to M_{r,d}$, which turns $\wt M^{\parab,0}_{r,d}$ into a relative flag variety over $M_{r,d}$ (indeed, a parabolic bundle $(\CE,\theta,(F_p))$ satisfying $\mathrm{res}_p(\theta)=0$ is stable if and only if the pair $(\CE,\theta)$ is stable in the usual sense). The codimension of $\wt M^{\parab,0}_{r,d}$ in $M_{r,d}^\parab$ equals to $\binom{r}{2}$. The relative dimension of $\wt M_{r,d}^{\parab,0}$ over $M_{r,d}$ is also $\binom{r}2$. Thus we have maps
\[
A=\pi_* \iota^*: H^*(M_{r,d}^{\parab,0}) \to H^{*-2\binom{r}2}(M_{r,d}),\qquad B=\iota_* \pi^*: H^*(M_{r,d}) \to H^{*+2\binom{r}2}(M_{r,d}^{\parab,0}).
\]
Let $\chi:M_{r,d}\to \BA$, $\chi^\parab:M^{\parab,0}_{r,d}\to \BA^{\parab,0}$ be the corresponding Hitchin maps, which are both projective. The space $\BA$ is identified with a linear subspace of $\BA^{\parab,0}$ so we may view both Hitchin morphisms as taking values in $\BA^{\parab,0}$.
Applying Proposition~\ref{prop:functoriality}, we see that both $A$ and $B$ preserve the perverse filtrations. Let
\[
\Delta \coloneqq \prod_{1\leq i<j\leq r} (y_{p,i} - y_{p,j})\in H^{2\binom{r}2}(M_{r,d}^{\parab,0}).
\]
We have $\pi_*(\Delta) = \pm r!$ because the class $\Delta$ is, up to a sign, the Euler class of the relative tangent bundle of $\wt M^{\parab,0}_{r,d}\to M_{r,d}$. The composition $\iota^* \iota_*$ is the multiplication by the Euler class of the normal bundle of $\wt M^{\parab,0}_{r,d}$ in $M_{r,d}^{\parab,0}$, which is also given by $\Delta$ up to a sign. 
So we obtain
\[
A B = \pi_* \iota^* \iota_* \pi^* = \pi_* \Delta \pi^* = \pm r!.
\]
As $\wt M^{\parab,0}_{r,d}\subset M^{\parab,0}_{r,d}$ is the zero set of a vector bundle with Euler class $\pm \Delta$, $\Delta$ has to be a multiple of the fundamental class $[\wt M^{\parab,0}_{r,d}]\in H^{2\binom{r}2}(M_{r,d}^{\parab,0})$, and from $\iota^*\iota_*=\pm\Delta$ we obtain $[\wt M^{\parab,0}_{r,d}]=\pm \Delta$.
\begin{thm}\label{thm:PW classical}
    $P=C$ holds for $M_{r,d}$, i.e the subspace $P_m H^*(M_{r,d})$ is the span of products $\prod_i \psi_{m_i}(\xi_i)$ satisfying $\sum_i m_i\leq m+N$, where $N=(g-1)r^2 +1$.
\end{thm}
\begin{proof}
    Let $-N'$ be the perversity of $1\in H^0(M^{\parab,0}_{r,d})$. We have $\pm r! = A(\Delta)\in P_{-N'+\binom{r}2} H^0(M_{r,d})$, so $-N\leq -N'+\binom{r}2$. Conversely, $\Delta = B(1) \in P_{-N}$. If $-N< -N'+\binom{r}2$, since $\Delta$ is a non-zero $\frakh$-homogeneous element of weight $-N'+\binom{r}2$, we have 
    \[
    N' = N + \binom{r}2.
    \]
    Consider a product of the form $f = \prod_i \psi_{m_i}(\xi_i)\in H^*(M_{r,d})$. Then
    \[
    B(f) = \pm \Delta f\in P_{\sum_i m_i -N'+\binom{n}2} = P_{\sum_i m_i -N}\quad\Rightarrow\quad f = \pm\frac{1}{r!}ABf\in P_{\sum_i m_i -N},
    \]
    so that $C\subset P$.
    Conversely, let $f\in P_m H^*(M_{r,d})$. By Markman's theorem, $f$ can be explicitly written as a polynomial in tautological classes. Write
    \[
    B(f) = \Delta f \in P_m H^*(M^{\parab,0}_{r,d}).
    \]
    By $P=C$ for $M^{\parab,0}_{r,d}$ we can write
    \[
    \Delta f = \sum_k \lambda_k g_k\qquad(\lambda_k\in\BC),
    \]
    where each $g_k$ is a monomial in tautological classes of the form $\prod_i \psi_{m_i}(\xi_i)$ with $\sum_i m_i\leq N'+m$. Using the identification of $H^*(M^{\parab,0}_{r,d})$ with the pure part of the cohomology of the corresponding moduli space $M^{\parell}_{r,d}$ we see that the symmetric group $\mathfrak{S}_r$ acts on $H^*(M^{\parab,0}_{r,d})$ permuting the generators $y_{p,r} = \psi_1(\tau_{p,r})$. Let $\ASym:H^*(M^{\parab,0}_{r,d}) \to H^*(M^{\parab,0}_{r,d})$ be the corresponding antisymmetrization operator,
    \[
    \ASym = \frac{1}{r!} \sum_{\sigma\in \mathfrak{S}_r} (-1)^{|\sigma|} \sigma.
    \]
    We obtain
    \[
    \Delta f = \sum_k \lambda_k \ASym(g_k).
    \]
    Each $\ASym(g_k)$ is explicitly a product of $\Delta$ and a linear combination of monomials of the form $\prod_i \psi_{m_i}(\xi_i)$ with $\sum_i m_i\leq N+m$ not containing the generators $y_{p,i}$. Therefore we can lift each of these and write
    \[
    B(f) = \Delta f = B\left(\sum_k \lambda_k' g_k'\right),\qquad \lambda_k'\in\BC,
    \]
    where each $g_k$ is a monomial in tautological classes of the form $\prod_i \psi_{m_i}(\xi_i)$ with $\sum_i m_i\leq N+m$.
    Applying $A$ we see that
    \[
    f = \sum_k \lambda_k' g_k',
    \]
    so the converse inclusion $P\subset C$ is proved.
\end{proof}

\begin{rem}
The number $N$, i.e. the negative of the perversity of $1$ in the above statements is always the dimension of the fibers of the Hitchin map.
\end{rem}

\begin{rem}\label{rem:H2-action}
	It is clear from the above proof that $H^*(M_{r,d})$ is the anti-invariant part of $H^*(\wt M^{\parab,0}_{r,d})$. The corresponding statement on the Betti side was established in \cite{mellit2019cell}. The anti-invariant part is clearly preserved by the operators $\wt D_{m,n}(\xi)_\red$ for $\xi\in H$; in particular, we have an action of $\CH_2$ on $H^*(M_{r,d})[x,y]$ as in Corollary \ref{cor:rational}.
\end{rem}

\subsection{The weight filtration}\label{ssec:weight}
Finally, let $X_{r,d}$ be the Betti moduli space, i.e. the character variety for a compact or non-compact curve $C$ with generic local monodromies, see \cite{hausel2011arithmetic}. There is a decomposition turning $H^*(X_{r,d})$ into a graded ring
\[
H^*(X_{r,d}) = \bigoplus_i W_{2i} H^*(X_{r,d}) \cap F^i H^*(X_{r,d}),
\]
where $W$ is the weight filtration and $F$ is the Hodge filtration. The classes $f\in W_{2i} H^*(X_{r,d}) \cap F^i H^*(X_{r,d})$ are said to be of \emph{pure weight} $i$. In the case of a compact curve it is explained in \cite{shende2016weights} that the standard tautological classes are of pure weight.
More precisely, let $\CE$ be the universal (topological) vector bundle on $X_{r,d}\times C$; then the Chern classes $e'_k(\xi) \coloneqq \int_C c_k(\CE)\cup\gamma$ are of pure weight $k$ for all $\gamma \in H^*(C)$.
The non-abelian Hodge theory provides a homeomorphism $X_{r,d} \simeq M_{r,d}$ which is compatible with the universal (topological) vector bundle.
In particular, the tautological classes above can be seen as Chern characters of the universal vector bundle $E$ on $M_{r,d}\times C$ instead.

\begin{lem}\label{lem:psi-GRR}
    We have $\psi_n(\gamma) = \frac{p'_n(\gamma)}{n+1} + \eps$, where $\eps$ is a linear combination of $p'_{i}(\gamma')$ with $i < n$, $\gamma'\in H^*(C)$.
\end{lem}
\begin{proof}
    Let $S = \mathrm{P}_C(\Omega(D) \oplus \CO)$, and $\pi:S\to C$ the natural projection.
    Recall that the tautological sheaf $F$ on $M_{r,d}\times S$, is related to the tautological vector bundle $E$ via $\pi_*(F) = E$.
    The relative Todd class of $\pi$ is given by $\Td_\pi = 1 + (\xi+(g-1)\omega)$, see Remark~\ref{rmk:ci-from-rd} for the notations.
    In particular, by Grothendieck-Riemann-Roch and projection formula we have
    \begin{align*}
        p'_k(\gamma) & = \int_C \ch_k (E)\cup \gamma = \int_S \ch_{k+1}(F)\cup \pi^*\gamma + \int_S \ch_k(F)\cup (\xi+(g-1)\omega)\cup \pi^*\gamma\\
        & = p_{k+1}(\pi^*\gamma) + p_k((\xi+(g-1)\omega)\cup \pi^*\gamma).
    \end{align*}
    On the other hand, by the definition~\eqref{eq:psi definition} of $\psi_n(\gamma)$ we have 
    \[
        \psi_n(\gamma) = \frac{p_{n+1}(\gamma)}{n+1} + \sum_{1\leq i\leq n} \frac{n!}{i!}p_i(\Td_{n+1-i}\gamma) = \frac{p'_n(\gamma)}{n+1} + \ldots,
    \]
    where dots stand for a linear combination of terms $p'_{i}(\gamma')$ with $i < n$,  $\gamma'\in H^*(C)$.
\end{proof}

Shende's result was extended to the parabolic character varieties in~\cite{mellit2019cell}.
Since the parabolic tautological classes are already expressed in terms of the vector bundle $E$, see Definition~\ref{def:Hrd} and the discussion afterwards, the analogue of Lemma~\ref{lem:psi-GRR} (without the factor $\frac{1}{n+1}$) holds for $\gamma \in H_{r,D}$.
Thanks to relations between Chern classes and Chern characters induced by~\eqref{eq:prod-of-sym-fts}, we see that
\begin{align*}
    \mathrm{Span} & \left\{\prod_i e'_{n_i}(\gamma_i) : \sum_i n_i \leq n\right\} = 
    \mathrm{Span}\left\{\prod_i p'_{n_i}(\gamma_i) : \sum_i n_i \leq n\right\} \\
    &=\mathrm{Span}\left\{\prod_i \psi_{n_i}(\gamma_i) : \sum_i n_i \leq n\right\}.
\end{align*}
Altogether, we obtain

\begin{prop}
The subspace $W_{2m} H^*(X_{r,d})$ is the span of products $\prod_i \psi_{m_i}(\xi_i)$ satisfying $\sum_i m_i\leq m$.	\qed
\end{prop}
Therefore the $P=C$ statements in Corollary~\ref{cor:parabolic} and Theorem~\ref{thm:PW classical} are equivalent to the claim $P_i=W_{2(i+N)}$.
\begin{cor}
    The $P=W$ conjecture holds both in classical and parabolic setup.\qed
\end{cor}

\section*{Acknowledgments}
We are thankful to Mark Andrea de Cataldo, Ben Davison, Eugene Gorsky, Lothar G\"ottsche, Andrei Negu{\c{t}}, Alexei Oblomkov, Raphaël Picovschi, Lev Rozansky, Kostya Tolmachov, Eric Vasserot for useful discussions and comments.

Tam\'as Hausel was partially supported by FWF grant “Geometry of the tip of the global nilpotent cone” no. P 35847.
Anton Mellit is supported by the consolidator grant No. 101001159 ``Refined invariants in combinatorics, low-dimensional topology and geometry of moduli spaces'' of the European Research Council.
Alexandre Minets was supported by the starter grant ``Categorified Donaldson-Thomas Theory'' No. 759967 of the European Research Council. 
Olivier Schiffmann is partially supported by the PNRR grant CF 44/14.11.2022 \textit{'Cohomological Hall algebras of surfaces and applications'} and would like to thank the Simion Stoilow Institute of Mathematics of the Romanian Academy for its hospitality and great working conditions. 

\bibliographystyle{amsalpha}

\end{document}